\documentclass[reqno]{amsart}
\usepackage{amssymb}
\usepackage{graphicx}

\usepackage[usenames, dvipsnames]{color}
\usepackage{verbatim}
\usepackage{mathrsfs}

\numberwithin{equation}{section}

\newtheorem{theorem}{Theorem}[section]
\newtheorem{corollary}[theorem]{Corollary}
\newtheorem{lemma}[theorem]{Lemma}
\newtheorem{proposition}[theorem]{Proposition}

\theoremstyle{definition}
\newtheorem{remark}[theorem]{Remark}

\theoremstyle{definition}

\theoremstyle{definition}

\makeatletter
\def\dashint{\operatorname%
{\,\,\text{\bf-}\kern-.98em\DOTSI\intop\ilimits@\!\!}}
\makeatother

\def\\det{\text{det}}

\def\.5{\frac{1}{2}}
\def\bA{\mathbb{A}}

\def\bR{\mathbb{R}}

\def\cA{\mathcal{A}}

\def\cD{\mathcal{D}}

\def\cP{\mathcal{P}}
\def\cM{\mathcal{M}}

\newcommand{\RN}[1]{%
  \textup{\uppercase\expandafter{\romannumeral#1}}%
}

\renewcommand{\epsilon}{\varepsilon}

\begin{document}
\title[Nonlocal parabolic equations]{On Schauder estimates for a class of nonlocal fully nonlinear parabolic equations}
\author[H. Dong]{Hongjie Dong}
\address[H. Dong]{Division of Applied Mathematics, Brown University,
182 George Street, Providence, RI 02912, USA}
\email{Hongjie\_Dong@brown.edu}
\thanks{H. Dong and H. Zhang were partially supported by the NSF under agreements DMS-1056737 and DMS-1600593.}

\author[H. Zhang]{Hong Zhang}
\address[H. Zhang]{Division of Applied Mathematics, Brown University,
182 George Street, Providence, RI 02912, USA}
\email{Hong\_Zhang@brown.edu}

\begin{abstract}
We obtain Schauder estimates for a class of concave fully nonlinear nonlocal parabolic equations of order $\sigma\in (0,2)$ with rough and non-symmetric kernels. We also prove that the solution to a translation invariant equation with merely bounded data is $C^\sigma$ in $x$ variable and $\Lambda^1$ in $t$ variable, where $\Lambda^1$ is the Zygmund space. From these results, we can derive the corresponding results for nonlocal elliptic equations with rough and non-symmetric kernels, which are new even in this case.
\end{abstract}
\maketitle

\section{introduction}
This paper is devoted to the study of Schauder estimates for a class of concave fully nonlinear nonlocal parabolic equations. There is a vast literature on Schauder estimates for classical elliptic and parabolic equations, for instance, see \cite{GT01,Kry97,CC95}.  Since the work by Caffarelli and Silvestre \cite{CS09,CSARMA,CS11}, nonlocal equations, which naturally arise from models in physics, engineering, and finance that involve long range interactions (for instance, see \cite{ContTankov}),  attract an increasing level of interest recently.
An example of  nonlocal operators, which is associated with pure jump processes (see, for instance, \cite{MP92}), is the following
\begin{align}
&L_au=\int_{\bR^d}\big(u(t,x+y)-u(t,x)-y^TDu(t,x)\big)K_a(t,x,y)\,dy\quad \text{for}\,\, \sigma\in (1,2),\nonumber\\
&L_au=\int_{\bR^d}\big(u(t,x+y)-u(t,x)-y^TDu(t,x)\chi_{B_1}\big)K_a(t,x,y)\,dy\quad  \text{for}\,\, \sigma=1\nonumber\\
&\quad\quad\text{with}\quad \int_{S_r}yK_a(t,x,y)\,ds=0\quad \forall\, r>0,\label{eq10.58} \\
&L_au=\int_{\bR^d}\big(u(t,x+y)-u(t,x)\big)K_a(t,x,y)\,dy\quad \text{for} \,\, \sigma\in (0,1),\nonumber
\end{align}
where
\begin{equation*}
K_a\in \mathcal{L}_0:=\Big\{K: \frac{\lambda(2-\sigma)}{|y|^{d+\sigma}}\le K(t,x,y)\le\frac{\Lambda(2-\sigma)}{|y|^{d+\sigma}}\Big\}
\end{equation*}
for some ellipticity constants $0<\lambda\le \Lambda$, with no regularity assumption imposed with respect to the $y$ variable.
This type of nonlocal operators was first considered by Komatsu \cite{MR746702}, Mikulevi$\check{\text{c}}$ius and Pragarauskas \cite{MP92, MR3201992}, and later by Dong and Kim \cite{DKDCDS,DKJFA12}, and Schwab and Silvestre \cite{SS14}, to name a few. In particular, the condition \eqref{eq10.58} appeared in all these references except for \cite{SS14}, where a similar cancellation condition was imposed. Notice that this class of operators is scaling invariant.

The fully nonlinear nonlocal parabolic equation that we are interested in is of the form
\begin{equation}
u_t=\inf_{a\in \cA}(L_au+f_a),\label{eq 1}
\end{equation}
where $K_a\in \mathcal{L}_0$ for  $a\in \mathcal{A}$ and $\cA$ is an index set. For fully nonlinear second-order equations with $f_{a}\equiv 0$, the celebrated $C^{2,\alpha}$ estimate was established independently by Evans \cite{Evans82} and Krylov \cite{Kry82} in early nineteen-eighties. Nonhomogeneous second-order equations were considered a bit later by Safonov \cite{Saf88}. Recently, Caffarelli and Silvestre \cite{CS11} investigated the nonlocal version of Evans-Krylov theorem with translation invariant and symmetric kernels, i.e., $K_a(x,y)=K_a(y)=K_a(-y)$, satisfying additional regularity assumptions
\begin{equation}\label{eq 12.141}
[K_a]_{C^2(\bR^d\setminus B_\rho)}\le \Lambda(2-\sigma)\rho^{-d-\sigma-2}.
\end{equation}
More recently, their result was extended to nonhomogeneous fully nonlinear elliptic equations by Jin and Xiong \cite{Jin2015} by using a recursive Evans-Krylov theorem.
At almost the same time, Serra \cite{Serra15} removed the regularity assumption \eqref{eq 12.141} and proved the Evans-Krylov theorem and Schauder estimates with symmetric kernels. His proof relies on a Liouville type theorem and a blow-up analysis.
 In this paper, we do not assume that the kernels are symmetric, which is certainly more general than the kernels considered in \cite{CS11,Jin2015, Serra15}.
Specifically,  when the kernels are symmetric, \eqref{eq10.58}
is satisfied automatically, and
\begin{equation*}
L_au=\frac{1}{2}\int_{\bR^d}\big(u(x+y)+u(x-y)-2u(x)\big)K_a(y)\,dy,
\end{equation*}
which is the form of the operators considered in \cite{CS11, Jin2015, Serra15}.

For equations with non-symmetric kernels, Dong and Kim \cite{DKJFA12,DKDCDS} proved $L_p$ and Schauder estimates for linear elliptic equations. Chang-Lara and  D\'avila \cite{ChangD12,ChangD13} considered nonlocal parabolic equations with non-symmetric kernels and critical drift, and proved the corresponding $C^\alpha$ and $C^{1,\alpha}$ estimate. Recently in \cite{ChangD15}, they  proved a version of the Evans-Krylov theorem for concave nonlocal parabolic equations with critical drift, where they assumed the kernels to be non-symmetric but translation invariant and smooth \eqref{eq 12.141}. We also mention that Schauder estimates for linear nonlocal parabolic equations were studied in \cite{MR3393263,MR3201992}. 

The objective of this paper is twofold. First we extend the previous results in \cite{Serra15,ChangD15,Jin2015,MR3393263} to include concave nonlocal parabolic equations with non-symmetric rough kernels. More specifically, for any small $\alpha$, if $f_a$ and $K_a(t,x,y)$ are $C^\alpha$ in $x$ and $C^{\alpha/\sigma}$ in $t$, then we have the following $C^{1+\alpha/\sigma,\sigma+\alpha}$ a priori estimate of any smooth solution $u$ to \eqref{eq 1} in $(-1,0)\times B_1$.

\begin{theorem}\label{thm 1}
Let $\sigma\in (0,2)$, $0<\lambda\le \Lambda<\infty$, and $\cA$ be an index set. There is a constant $\hat{\alpha}\in (0,1)$ depending on $d,\sigma,\lambda$, and $\Lambda$ so that the following holds. Let $\alpha\in (0, \hat{\alpha})$ such that $\sigma+\alpha$ is not an integer.  Assume $K_a\in \mathcal{L}_0$ and satisfies \eqref{eq10.58} when $\sigma=1$, and
\begin{equation}
                                    \label{eq 11.051}
\big|K_a(t,x,y)-K_a(t',x',y)\big|\le A\big(|x-x'|^\alpha+|t-t'|^{\alpha/\sigma}\big)\frac{\Lambda(2-\sigma)}{|y|^{d+\sigma}},
\end{equation}
where $A\ge 0$ is a constant.
 Suppose $u\in C^{1+\alpha/\sigma,\sigma+\alpha}(Q_1)\cap C^{\alpha/\sigma,\alpha}((-1,0)\times\bR^d)$ is a solution of
\begin{equation}
                                \label{eq 1110.1}
u_t=\inf_{a\in \mathcal{A}}(L_a u+f_a)\quad \text{in}\,\, Q_1,
\end{equation}
where $f_a\in C^{\alpha/\sigma,\alpha}(Q_1)$ satisfying
$$
C_0:=\sup_{a\in \cA}[f_a]_{\alpha/\sigma,\alpha;Q_1}<\infty,\quad
\sup_{(t,x)\in Q_1}\Big|\inf_{a\in \mathcal{A}}f_a(t,x)\Big|<\infty.
$$
Then,
\begin{equation}\label{eq 12.142}
[u]_{1+\alpha/\sigma,\alpha+\sigma;Q_{1/2}}\le C\|u\|_{\alpha/\sigma,\alpha;(-1,0)\times\bR^d}+CC_0,
\end{equation}
where $C>0$ is a constant depending only on $d,\lambda,\Lambda,\alpha, A,$ and $\sigma$, and is uniformly bounded as $\sigma\to 2$.
\end{theorem}
In the theorem above, $\|\cdot\|_{\alpha/\sigma,\alpha;\Omega}$ is the H\"older norm of order $\alpha/\sigma$ in $t$ and $\alpha$ in $x$ with underlying domain $\Omega$. We used $Q_r$ to denote the parabolic cylinder with radius $r$ centered at the origin. For precise definitions, see Section 2. As pointed out in \cite{Serra15}, the $C^{\alpha/\sigma,\alpha}$ H\"older norm of $u$ on the right-hand side of \eqref{eq 12.142} is necessary and cannot be replaced by the $L_\infty$ norm or any lower-order H\"older norm of $u$. We also note that by keeping track of the constants in the proofs below, in the symmetric case, if $\sigma \in [\sigma_0, 2)$ for some $\sigma_0 \in (0,1)$, then
the constant $C$ in \eqref{eq 12.142} depends on $\sigma_0$, not $\sigma$.
In the non-symmetric case, if $0 < \sigma_0 \le \sigma \le \sigma_1 < 1$
(or $1 < \sigma_2 \le \sigma < 2$), then the constant $C$ depends on $\sigma_0$ and $\sigma_1$ (or $\sigma_2$), not $\sigma$. In particular, $C$ does not blow up as $\sigma$ approaches $2$.

Roughly speaking, the proof of Theorem \ref{thm 1} can be divided into three steps.
First we prove a Liouville type theorem for solutions in $(-\infty,0)\times \bR^d$.  For the classical PDEs, we generally apply interpolation and iteration to obtain $C^{1,\alpha}$ and $C^{2,\alpha}$ estimates. 
One notable feature of nonlocal operators is that the boundary data is prescribed on the complement of the domain where the equation is satisfied, which makes it difficult to implement these techniques. 
However, if we assume that \eqref{eq 1} is satisfied in $(-\infty,0)\times\bR^d$, then we do not need to handle boundary data any more, which is the advantage of considering an equation satisfied in the whole space.
Second, we prove the a priori estimate for equations with translation invariant kernels by combining the Liouville theorem and a blow-up analysis. Particularly in this step, the extension from symmetric kernels to non-symmetric kernels is non-trivial. A key idea in the classical Evans-Krylov theorem for $F(D^2u)=0$ is that, since the function $F$ is concave, any second directional derivative $D_{ee}^2u$ is a subsolution. It is relatively easy to adapt this idea to the nonlocal equations with symmetric kernels due to the appearance of centered second order difference in the definition of the operator. For nonsymmetric kernels,
some new ideas are required to obtain a similar subsolution as in the symmetric case.  Moreover, the dependence of the $t$ variable also makes the proof more involved.
 Finally, we implement a more or less standard perturbation argument to treat the general case.

The second objective of this paper is to consider the end-point situation when $\alpha=0$.  For second-order elliptic equations, even the Poisson equation $\Delta u=f$, when $f$ is merely bounded, it is well known that $u$ may fail to be $C^{1,1}$. However, this is not the case for nonlocal equations. When $\sigma\neq 1$ and the kernels are independent of $t$ and $x$, we prove a priori $C^\sigma$ estimate in the $x$ variable and  $\Lambda^1$ estimate in the $t$ variable when $f_a$ is merely bounded and measurable, where $\Lambda^1$ is the Zygmund space. When $\sigma=1$, we obtain a priori $\Lambda^1$ estimate in both  $t$ and $x$. 
We assume that the solution is smooth because the spaces in which estimates are obtained are not fine enough
for the nonlocal operators to be defined pointwise.

\begin{theorem}\label{thm 2}
(i) Let $\sigma\neq 1$. Assume that $u$ is 
a smooth solution to \eqref{eq 1} in $(-\infty,0)\times\bR^{d}$ with $K_a$ independent of $t$ and $x$. When $\sigma\in (1,2)$, we also assume that $Du$ is $C^{(\sigma-1)/\sigma}$ in $t$. Then there exists a constant $C$ depending on $d,\lambda,\Lambda$, and $\sigma$ such that for $\sigma>1$,
\begin{equation*}
[u]^t_{\Lambda^1}+[u]^*_{\sigma}+[Du]^t_{\frac{\sigma-1}{\sigma}}\le C\sup_a\|f_a\|_{L_\infty};
\end{equation*}
and for $\sigma<1$,
 \begin{equation*}
[u]^t_{\Lambda^1}+[u]^*_{\sigma}\le C\sup_a\|f_a\|_{L_\infty}.
\end{equation*}

(ii) Let $\sigma=1$. Assume that $u$ is 
a smooth solution to \eqref{eq 1} in $(-\infty,0)\times\bR^{d}$ with $K_a$ independent of $t$ and $x$. Then there exists a constant $C$ depending on $d,\lambda,\Lambda$, and $\sigma$ such that
 \begin{equation*}
[u]_{\Lambda^1}\le C\sup_a\|f_a\|_{L_\infty}.
\end{equation*}
Here all the norms are taken in $\bR^{d+1}_0:=(-\infty,0)\times \bR^d$.
\end{theorem}

In the theorem above, $[\,\cdot\,]^\ast$, $[\,\cdot\,]^t_{\alpha}$, $[\,\cdot\,]_{\Lambda^1}^t$, and $[\,\cdot\,]_{\Lambda^1}$ are the H\"older semi-norms in $x$, $t$, the Zygmund semi-norm in $t$, and the Zygmund semi-norm with respect to $(t,x)$, respectively. See the precise definitions in Section 2.

We localize Theorem \ref{thm 2} to obtain the following corollary.
\begin{corollary}\label{cor 1.3}
(i) Let $\sigma\neq 1$. Assume that $u$ is a smooth solution to \eqref{eq 1} with $K_a$ independent of $t$ and $x$.
Then for $\sigma>1$,
\begin{equation*}
[u]^\ast_{\sigma;Q_{1/2}}+[u]^t_{\Lambda^1(Q_{1/2})}+[D u]^t_{\frac{\sigma-1}{\sigma};Q_{1/2}}\le C\Big(\sup_{a\in \mathcal{A}}\|f_a\|_{L_\infty(Q_1)}+\|u\|_{L_\infty((-1,0)\times \bR^d)}\Big);
\end{equation*}
and for $\sigma<1$,
\begin{equation*}
[u]^\ast_{\sigma;Q_{1/2}}+[u]^t_{\Lambda^1(Q_{1/2})}
\le C\Big(\sup_{a\in \mathcal{A}}\|f_a\|_{L_\infty(Q_1)}+\|u\|_{L_\infty((-1,0)\times \bR^d)}\Big).
\end{equation*}

(ii) Let $\sigma=1$. Assume that $u$ is a smooth solution to \eqref{eq 1110.1} with $K_a$ independent of $t$ and $x$. Then we have
\begin{equation*}
[u]_{\Lambda^1;Q_{1/2}}
\le C\Big(\sup_{a\in \mathcal{A}}\|f_a\|_{L_\infty(Q_1)}+\|u\|_{L_\infty((-1,0)\times \bR^d)}\Big).
\end{equation*}
\end{corollary}

To our best knowledge, such result is new even for nonlocal elliptic equations with symmetric kernels. A similar result was obtained very recently by Mou \cite{Mo16} for elliptic equations with symmetric, smooth kernels and Dini continuous data. With merely bounded and measurable data, the best estimates known in the literature are the $C^\beta$ regularity of $u$ in $t$ for $\beta<1$ and the $C^\gamma$ regularity of $u$ in $x$ for $\gamma<\min\{\sigma,1+\alpha\}$, where $\alpha>0$ is a small constant. See \cite{Se15a,ChKr15}. Because $\Lambda^1\subsetneq C^\beta$ for any $\beta<1$, Theorem \ref{thm 2} improves these results and is optimal even in the linear case. See Remark \ref{rem1.02} below.

The proof of Theorem \ref{thm 2} is based on a perturbation type argument using Campanato's approach. We first refine the estimate in Theorem \ref{thm 1} when the operator is translation invariant. In particular, we replace  $\|u\|_{\alpha/\sigma,\alpha;(-1,0)\times \bR^d}$ on the right-hand side of \eqref{eq 12.142} by
\begin{equation}
                            \label{eq10.46}
[u]_{\alpha/\sigma,\alpha;(-1,0)\times B_2}+\sum_{j=2}^\infty2^{-j\sigma}[u]_{\alpha/\sigma,\alpha;(-1,0)\times (B_{2^j}\setminus B_{2^{j-1}})}.
\end{equation}
The advantage of the replacement  will be explained below. Another important ingredient in the proof is  the fact that, for example, when $\sigma>1$,
\begin{equation}
[u]^t_{\Lambda^1(\bR^{d+1}_0)}+[u]^\ast_{\sigma;\bR^{d+1}_0}
+[Du]^t_{\frac{\sigma-1}{\sigma};\bR^{d+1}_0}
\le C\sup_{r>0}\sup_{(t,x)\in \bR^{d+1}_0}E[u;Q_r(t,x)], \label{eq 12.151}
\end{equation}
where
\begin{equation*}
E[u;Q_r(t,x)]:=\inf_{p\in \mathcal{P}_1}r^{-\sigma}\|u-p\|_{L_\infty(Q_r(t,x))},
\end{equation*}
$\mathcal{P}_1$ is the set of linear functions in $(t,x)$, and  $Q_r(t,x)$ is the parabolic cylinder with center $(t,x)$; see \eqref{eq 12.281}. Therefore, instead of directly estimating $$
[u]^t_{\Lambda^1}+[u]^\ast_\sigma+[Du]^t_{\frac{\sigma-1}{\sigma}},
$$
we estimate $E[u;Q_r(t,x)]$ for any fixed $r$ and $(t,x)$.  It is worth noting that in view of the proofs of Lemmas \ref{lem5.2}, \ref{lem5.5b}, and \ref{lem5.4}, the two quantities on the left and right hand sides of \eqref{eq 12.151} are actually equivalent.

More specifically, without loss of generality, we set $(t,x)=(0,0)$ and let $v_K$ solve the homogeneous equation
\begin{align*}
\begin{cases}
\partial_t v_K=\inf_{a\in\mathcal{A}}L_av_K\quad &\text{in}\,\, Q_{2R}\\
v_K=g_K:=\max\{-K,\min\{u-p,K\}\}\quad &\text{in}\,\, (-(2R)^\sigma,0)\times B_{2R}^c
\end{cases},
\end{align*}
where $K$ is a large constant, $p$ is a carefully chosen linear function, and $R>2r$ is a constant to be determined.  Now we apply Theorem \ref{thm 1} to $v_K$ and control $[v_K]_{1+\alpha/\sigma,\alpha+\sigma;Q_{R/2}}$ by using scaling argument and replacing $\|v_K\|_{\alpha/\sigma,\alpha}$ by \eqref{eq10.46}. It is easily seen that in each cylindrical domain $(-R^\sigma,0)\times(B_{2^jR}\setminus B_{2^{j-1}R})$, the H\"older norm of $g_K$ is bounded and independent of $K$, but globally it depends on $K$ and goes to infinity as $K\to \infty$. This is also the advantage of decomposing the domain into annuli.  
We then set $q_K$ to be the first-order Taylor expansion of $v_K$ and estimate
\begin{align*}
\|u-p-q_K\|_{L_\infty(Q_r)}\le \|u-p-v_K\|_{L_\infty(Q_r)}+\|v_K-q_K\|_{L_\infty(Q_r)},
\end{align*}
where the first term is bounded by $CR^\sigma$ due to the Aleksandrov-Bakelman-Pucci estimate and second term is controlled  by $[v_K]_{1+\alpha/\sigma,\alpha+\sigma;Q_r}$. Finally, we are able to obtain 
\begin{equation*}
r^{-\sigma}\|u-p-q_K\|_{L_\infty(Q_r)}\le C(r/R)^\alpha\big([u]^\ast_\sigma+[u]^t_{\Lambda^1}
+[Du]^t_{\frac{\sigma-1}{\sigma}}\big)+C(R/r)^\sigma\|f\|_{L_\infty}.
\end{equation*}
By setting $R=Mr$, using \eqref{eq 12.151}, and taking $M$ sufficiently large, the terms involving $u$ on the right-hand side above are absorbed in the left-hand side.


\begin{remark}
                    \label{rem1.02}
We give two simple examples which indicate that the estimates in Theorem \ref{thm 2} (and thus in corollary \ref{cor 1.3}) are optimal even in the linear case. Set $d=1$. Let $f=f(t,x)=\chi_{\{t<-M|x|^\sigma\}}$ for some constant $M>0$ and $u$ be a solution to the equation
$$
u_t+(-\Delta)^{\sigma/2}u=f\quad\text{in}\,\,(-\infty,0)\times \bR^{d}.
$$
Then by the explicit representation of solutions, it is easily seen that for sufficiently large $M$,
$$
\lim_{t\nearrow 0} u_t(t,0)=-\infty.
$$
Thus $u$ cannot be Lipschitz in $t$ in this case.
Next we set $\sigma=1$, $g=g(t,x)=\chi_{\{t<0,x>t\}}$, and $v$ be a solution to
$$
v_t+(-\Delta)^{1/2}u=g\quad\text{in}\,\,(-\infty,0)\times \bR^{d}.
$$
Again by the explicit representation of solutions,
$$
\lim_{x\to 0} v_x(0,x)=\infty.
$$
Therefore, $v$ cannot be Lipschitz in $x$ in this case.
\end{remark}

By viewing solutions to elliptic equations as steady state solutions to parabolic equations, from Theorems \ref{thm 1}, \ref{thm 2}, and Corollary \ref{cor 1.3}, we obtain the corresponding results for nonlocal elliptic equations with nonsymmetric and rough kernels.

The organization of this paper is as follows. In the next section, we introduce some notation and preliminary results that are necessary in the proof of our main results. We prove the Liouville theorem in Section 3 and Theorem \ref{thm 1} in Section 4. In Section 5, we apply Theorem \ref{thm 1} to prove Theorem \ref{thm 2}.

{\bf Remark added after the proof.} After we finished the paper, we learnt that Chang-Lara and Kriventsov \cite{ChKr15} also established a Schauder estimates for fully nonlinear nonlocal parabolic equations with rough kernels, by using a different method. In their paper, it is assumed that the kernels are time-independent and symmetric. They also obtained a $C^\beta$ estimate of $u$ in $t$ for $\beta<1$ when the data is bounded and measurable.

\section{Notation and preliminary results}
In this section, we introduce some notation which will be used throughout this paper and some preliminary results which are useful in our proof.
 We use $B_r(x)$ to denote the Euclidean ball in $\bR^d$ with center $x$ and radius $r$.
The parabolic cylinder $Q_r(t,x)$ is defined as follows
\begin{equation}
Q_r(t,x)=(t-r^\sigma,t)\times B_r(x).\label{eq 12.281}
\end{equation}
We simply use $Q_r$ to denote $Q_r(0,0)$ and $\bR^{d+1}_0:=(-\infty,0)\times\bR^d$.
Let $\Omega\subset \bR^{d+1}$ and we define the H\"older semi-norm as follows: for any $\alpha,\beta\in (0,1]$, and function $f$,
\begin{equation*}
[f]_{\beta,\alpha;\Omega}
=\sup\Big\{\frac{|f(t,x)-f(s,y)|}{\max(|x-y|^\alpha,|t-s|^{\beta})}: (t,x), (s,y)\in \Omega, (t,x)\neq (s,y)\Big\}.
\end{equation*}
We denote
\begin{equation*}
\|f\|_{\beta,\alpha;\Omega}=\|f\|_{L_\infty(\Omega)}+[f]_{\beta,\alpha;\Omega}.
\end{equation*}
For any nonnegative integers $m$ and $n$,
\begin{equation*}
\|f\|_{n+\beta,m+\alpha;\Omega}=\|f\|_{L_\infty(\Omega)}+[D^m f]_{\beta,\alpha;\Omega}+[\partial_t^n f]_{\beta,\alpha;\Omega}.
\end{equation*}
The spaces corresponding to $\|\cdot\|_{\alpha,\beta;\Omega}$ and $\|\cdot\|_{m+\alpha,n+\beta;\Omega}$ are denoted by $C^{\alpha,\beta}(\Omega)$ and $C^{m+\alpha,n+\beta}(\Omega)$, respectively.
Next, for any $\alpha,\beta\in (0,1]$, we define the H\"older semi-norms only with respect to $x$ or $t$
\begin{align*}
[f]^\ast_{\alpha;\Omega}
=\sup\Big\{\frac{|f(t,x)-f(t,y)|}{|x-y|^\alpha}: (t,x), (t,y)\in \Omega, x\neq y\Big\},\\
[f]^t_{\beta;\Omega}=\sup\Big\{\frac{|f(t,x)-f(s,x)|}{|t-s|^\beta}: (t,x), (s,x)\in \Omega, t\neq s\Big\}.
\end{align*}
When $\sigma=k+\alpha$ with some integer $k\ge 1$,
$$[f]^\ast_{\sigma;\Omega}=[D^k f]^\ast_{\alpha;\Omega}.$$
For $\alpha\in (0,2)$, we define the Lipschitz-Zygmund semi-norm and norm by
\begin{align*}
[u]_{\Lambda^\alpha}&=\sup_{|h|>0}|h|^{-\alpha}\|u(\cdot+h)
+u(\cdot-h)-2u(\cdot)\|_{L_\infty},\\
\|u\|_{\Lambda^\alpha}&=\|u\|_{L_\infty}+[u]_{\Lambda^\alpha}.
\end{align*}
We say $u\in \Lambda^\alpha$ if $\|u\|_{\Lambda^\alpha}<\infty$.

For simplicity of notation, we denote
\begin{align*}
\delta u(t,x,y)=
\begin{cases}
u(t,x+y)-u(t,x)-y^TDu(t,x)\quad&\text{for}\,\, \sigma\in(1,2),\\
u(t,x+y)-u(t,x)-y^TDu(t,x)\chi_{B_1}\quad&\text{for}\,\, \sigma=1,\\
u(t,x+y)-u(t,x)\quad &\text{for}\,\, \sigma\in (0,1).
\end{cases}
\end{align*}
The Pucci extremal operator is defined as follows: for $\sigma\neq 1$
\begin{align*}
\cM^+u(t,x)=\int_{\bR^d}\big(\Lambda\delta u(t,x,y)^+-\lambda\delta u(t,x,y)^-\big)\frac{2-\sigma}{|y|^{d+\sigma}}\,dy,\\
\cM^-u(t,x)=\int_{\bR^d}\big(\lambda\delta u(t,x,y)^+-\Lambda\delta u(t,x,y)^-\big)\frac{2-\sigma}{|y|^{d+\sigma}}\,dy.
\end{align*}
When $\sigma= 1$, the extremal operator cannot be written out explicitly, due to the condition \eqref{eq10.58}.
Nevertheless, we do not use exact representation directly and define the extremal operator by
\begin{equation*}
\cM^+u=\sup_{a}L_au\quad\quad \text{and}\quad \cM^-u=\inf_aL_a u,
\end{equation*}
where the infimum (or supremum) is taken with respect to all $L_a$'s with kernels $K_{a}$ satisfying \eqref{eq10.58}.

We recall the weak Harnack inequality of \cite[Theorem 6.1]{SS14}.
\begin{proposition}\label{thm 8.311}
Assume that $0<\sigma_0\le \sigma<2$ and $C>0$ is a constant. Let $u$ be a function such that
$$
u_t-\cM^-u\ge -C\quad \text{in}\quad Q_1,\quad
u\ge 0\quad \text{in}\quad (-1,0)\times \bR^d.
$$
Then there are constants $C_1>0$ and $\epsilon_1\in (0,1)$  depending only on $\sigma_0, \lambda, \Lambda$, and $d$, such that
$$
\Big(\int_{(-1,-2^{-\sigma})\times B_{1/4}}u^{\epsilon_1}\,dx\,dt\Big)^{1/\epsilon_1}\le C_1\Big(\inf_{Q_{1/4}}u+C\Big).
$$
\end{proposition}
From Proposition \ref{thm 8.311}, we obtain the following corollary for $\sigma\in(1,2)$, the proof of which is provided in the appendix.

\begin{corollary}\label{weak_harnack}
Let $\sigma_2\in (1,2)$, $\sigma\in (\sigma_2,2)$, $C>0$ be constants, and $u$ satisfy
\begin{align*}
u_t-\cM^-u\ge -C\quad \text{in}\quad Q_{2r},\quad u\ge 0\quad\text{in}\quad  (-(2r)^\sigma,0)\times \bR^d.
\end{align*}
Let $\varepsilon_1$ be the constant in Proposition \ref{thm 8.311}.
For any $r,\delta\in (0,1)$,  denote $\tilde{Q}_{\delta r}=(-r^\sigma,-(\delta r)^\sigma)\times B_r$. Then we have
\begin{equation*}
r^{-(d+\sigma)/\epsilon_1}\Big(\int_{\tilde{Q}_{\delta r}}u^{\epsilon_1}\,dx\,dt\Big)^{1/\epsilon_1}
\le C_2\Big(\inf_{Q_{\delta r/2}}u+Cr^\sigma\Big),
\end{equation*}
where $C_2>0$ is a constant depending only on $\delta$, $\sigma_2$, $\lambda$, $\Lambda$, and $d$.
\end{corollary}

We state the following local boundedness estimate from \cite[Corollary 6.2]{ChangD14}.

\begin{proposition}\label{thm 8.261}
 Let $\Omega\subset \bR^d$, $t_1<t_2$,  and $u$ satisfy
\begin{equation*}
u_t-\cM^+u\le 0\quad \text{in}\quad (t_1,t_2]\times \Omega.
\end{equation*}
Then for any $(t_1^\prime,t_2]\times \Omega^\prime\subset\subset (t_1,t_2]\times\Omega$,
\begin{equation*}
\sup_{\Omega^\prime\times(t_1^\prime,t_2]}u^+\le C(2-\sigma)\int_{t_1}^{t_2}\int_{\bR^d}\frac{u^+}{1+|x|^{d+\sigma}}\,dx\,dt,
\end{equation*}
where $C$ depends on $\Omega$, $\Omega'$, $t_1$, $t_2$, and $t_1'$.
\end{proposition}

Let us point out that the kernels considered in \cite{ChangD14} are more general than our kernels. Specifically, Chang-Lara and D\'avila considered  when $\sigma\in[1,2)$
\begin{align*}
Lu=(2-\sigma)\int_{\bR^d}\hat{\delta}u(x,y)K(y)\,dy+b\cdot Du(x),
\end{align*}
where $\hat{\delta}u(x,y)=u(x+y)-u(x)-D u(x)y\chi_{B_1}$, $K(y)\in \mathcal{L}_0$, and for some $\beta>0$,
\begin{align}\label{eq 12.301}
\sup_{r\in (0,1)}r^{\sigma-1}\Big|b+(2-\sigma)\int_{B_1\setminus B_r}yK(y)\,dy\Big|\le \beta.
\end{align}
Note that for $\sigma>1$, since
\begin{equation*}
\delta u(x,y)=\hat{\delta} u(x,y)-Du(x)y\chi_{B_1^c},
\end{equation*}
we can rewrite our operator and get
\begin{equation*}
b=-(2-\sigma)\int_{B_1^c}yK(y)\,dy.
\end{equation*}
Obviously, $|b|\le C$, where $C$ depends $d$, $\sigma$, and $\Lambda$,  and it is easy to check that \eqref{eq 12.301} holds for $b$ and $K$ above.

The next proposition is  \cite[Theorem 7.1]{SS14}.
\begin{proposition}\label{thm 10.131}
Let $0<\sigma_0\le \sigma<2$ and $u$ satisfy in $Q_1$
\begin{equation*}
u_t-\cM^+u\le C_0\quad\text{and}\quad
u_t-\cM^-u\ge -C_0.
\end{equation*}
Then there are constants $\gamma\in (0,1)$ and $C>0$ only depending on $d$, $\sigma_0$, $\lambda$, and $\Lambda$ such that
\begin{equation*}
[u]_{\gamma/\sigma,\gamma;Q_{1/2}}\le C\|u\|_{L_\infty((-1,0); L_1(\omega_\sigma))}+CC_0.
\end{equation*}
\end{proposition}
Here
\begin{equation*}
\|u\|_{L_\infty((-1,0);L_1(\omega_\sigma))}
=(2-\sigma)\sup_{t\in(0,1)}\int_{\bR^d}\frac{|u(t,x)|}{1+|x|^{d+\sigma}}\,dx.
\end{equation*}
Note that we replaced $\|u\|_{L_\infty((-1,0)\times \bR^d)}$ by $\|u\|_{L_\infty((-1,0); L_1(\omega_\sigma))}$, which follows from a simple localization argument. See, for instance, \cite[Corollary 7.1]{ChangD14}. In the sequel, we always assume $\gamma<\sigma$.

We finish this section by proving the following global H\"older estimate.
\begin{lemma}\label{lemma 12.271}
Let  $u$ satisfy in $Q_1$
\begin{equation*}
u_t-\cM^+u\le C_0\quad\text{and}\quad
u_t-\cM^-u\ge -C_0,
\end{equation*}
where $C_0$ is a constant and $u\equiv 0$ in  $\bR^{d+1}_0\setminus Q_1$. Then there exists a constant $\alpha\in (0,1)$ depending on $d,\lambda,\Lambda$, and $\sigma$ (uniformly as $\sigma\to 2$), so that
\begin{equation*}
[u]_{\alpha/\sigma,\alpha;Q_1}\le CC_0,
\end{equation*}
where $C$ depends on $d,\lambda,\Lambda$, and $\sigma$, which is uniformly bounded as $\sigma\to 2$.
\end{lemma}
\begin{proof}
Thanks to the interior H\"older estimate Proposition \ref{thm 10.131}, it suffices to prove the estimate near the parabolic boundary of $Q_1$. We consider the lateral boundary and bottom separately.
Define $\phi:\bR^d\to \bR^+$ as
\begin{equation*}
\phi(x)=\begin{cases}
x_d^\beta \quad&\text{for}\,\, x_d>0\\
0\quad &\text{for}\,\, x_d\le 0
\end{cases},
\end{equation*}
where $\beta\in(0,1)$. We claim that for sufficiently small $\beta\in (0,\sigma_0)$  depending on $d$, $\lambda$, $\Lambda$, and $\sigma_0$, we have
\begin{equation*}
\cM^+\phi(x)<-\hat{C}x_d^{\beta-\sigma}\quad \text{in}\quad \{x_d>0\},
\end{equation*}
where $\hat{C}$ depends on $d,\lambda,\Lambda$, and $\sigma_0$. By scaling,  it is obvious that
\begin{equation*}
\cM^+\phi(x)=x_d^{\beta-\sigma}\cM^+\phi(e),
\end{equation*}
where $e=(0,0,\ldots,0,1)$. Therefore, we only need to estimate $\cM^+\phi(e)$.

Case 1: $\sigma>1$.
 By definition,
\begin{align*}
&\cM^+\phi(e)\\
&=(2-\sigma)x_d^{\beta-\sigma}\left(\int_{\{y_d>-1\}}+\int_{\{y_d<- 1\}}\right)\big(\Lambda(\delta \phi(e,y))^+-\lambda(\delta \phi(e,y))^-\big)\frac{1}{|y|^{d+\sigma}}\,dy\\
&=:(2-\sigma)x_d^{\beta-\sigma}(\RN{1}_1+\RN{1}_2).
\end{align*}
When $y_d>-1$, by concavity, it follows that
\begin{equation*}
\phi(e+y)=(1+y_d)^\beta<1+\beta y_d\quad\text{and}\quad\delta \phi(e,y)\le 0.
\end{equation*}
Therefore,
\begin{align*}
\RN{1}_1&\le \int_{\{|y_d|<1\}}\lambda\delta \phi(e,y)\frac{1}{|y|^{d+\sigma}}\,dy\\
&=\lambda\int_{y_d\in(0,1)}\big((1+y_d)^{\beta}+(1-y_d)^\beta-2\big)\frac{1}{|y|^{d+\sigma}}\,dy.
\end{align*}
Notice that for any $s\in (-1,1)$,
$$(1+s)^\beta<1+\beta s+\frac{\beta(\beta-1)}{2}s^2+\frac{\beta(\beta-1)(\beta-2)}{6}s^3,$$
which implies that for $y_d\in (0,1)$
\begin{equation*}
(1+y_d)^{\beta}+(1-y_d)^\beta-2< \beta(\beta-1)y_d^2.
\end{equation*}
Therefore,
\begin{align*}
\RN{1}_1< \lambda\beta(\beta-1)\int_{y_d\in(0,1)}\frac{y_d^2}{|y|^{d+\sigma}}\,dy
=C_1\frac{\beta(\beta-1)}{2-\sigma},
\end{align*}
where $C_1$ depends on $d$ and $\lambda$.

Now we turn to $\RN{1}_2$. Since $\phi(e+y)=0$ when $y_d<-1$, we have
\begin{align*}
\RN{1}_2&=\int_{\{y_d<-1\}}\frac{\Lambda(-\beta y_d-1)^+-\lambda(-\beta y_d-1)^-}{|y|^{d+\sigma}}\,dy\\
&\le \int_{\{y_d<-1/\beta\}}\frac{\Lambda(-\beta y_d-1)}{|y|^{d+\sigma}}\,dy = C_2\beta^{\sigma},
\end{align*}
where $C_2$ depends on $\Lambda,d$, and $\sigma$, and is uniformly bounded as $\sigma\to 2$.
Thanks to the estimates of $\RN{1}_1$ and $\RN{1}_2$ above, it follows that
\begin{equation*}
\cM^+\phi(e)\le C_1\beta(\beta-1)+C_2\beta^\sigma.
\end{equation*}
By choosing $\beta$ sufficiently small depending on $\lambda$, $\Lambda$, $d$,  and $\sigma$ (but uniformly as $\sigma\to 2$) so that
$$C_1(\beta-1)+C_2\beta^{\sigma-1}\le -C_1/2,$$
the claim is proved.

Case 2: $\sigma< 1$.
Let $\RN{1}_1$ and $\RN{1}_2$ be defined as before.
Since $\phi(x)=0$ for $x_d<0$, we get
\begin{align}
\RN{1}_2=-(2-\sigma)\int_{\{y_d<-1\}}\frac{\lambda}{|y|^{d+\sigma}}\,dy=-C_3,\label{eq 1.141}
\end{align}
where $C_3>0$ depends on $\sigma,\lambda$, and $d$.
For $\RN{1}_1$, we have
\begin{align*}
\RN{1}_1&=(2-\sigma)\int_{\{y_d>0\}}\frac{\Lambda((1+y_d)^\beta-1)}{|y|^{d+\sigma}}\,dy
-(2-\sigma)\int_{y_d\in(-1,0)}\frac{\lambda(1-(1+y_d)^\beta)}{|y|^{d+\sigma}}\,dy\\ &\le (2-\sigma)\Lambda\int_{\{y_d>0\}}\frac{(1+y_d)^\beta-1}{|y|^{d+\sigma}}\,dy
\to 0\quad \text{as}\quad \beta\to 0
\end{align*}
by the monotone convergence theorem.
Therefore, we can choose $\beta$ small depending on $\Lambda$, $\lambda$, $d$, and $\sigma$ so that
\begin{equation*}
\cM^+\phi(e)\le -C_3/2.
\end{equation*}
The claim is proved.

Case 3: $\sigma=1$. In this case, we still have \eqref{eq 1.141}. For $\RN{1}_1$, we notice that integrand in the region $\{-1<y_d<0\}\cup\{|y|<1\}$ is negative, and
$$
\int_{\{y_d>0\}\cap \{|y|>1\}}\frac{(1+y_d)^\beta-1}{|y|^{d+\sigma}}\,dy
\to 0\quad \text{as}\quad \beta\to 0
$$
by the monotone convergence theorem. Thus the claim follows as well.

Now we are ready to consider $u$ near the lateral boundary. By a translation and rotation of the coordinates, we replace the ball $B_1$ by $B_1(e)$ and estimate $u$ near the origin. Define the barrier function $\psi(t,x)=\frac{C_0}{2^{\beta-\sigma}\hat{C}}\phi(x)$. Obviously,
\begin{equation*}
\partial_t \psi(t,x)-\cM^+\psi(t,x)=-\frac{C_0}{2^{\beta-\sigma}\hat{C}}\cM^+\phi(x)
>\frac{C_0}{2^{\beta-\sigma}}x_d^{\beta-\sigma}>C_0
\end{equation*}
when $x\in B_1(e)$.
On the other hand, $\psi\ge 0$ in $\bR\times\bR^d$. Since
 \begin{equation*}
 u_t-\cM^+u\le C_0
 \end{equation*}
and $u\equiv 0$ outside $Q_1$,  by the comparison principle,
\begin{equation*}
  u(t,x)\le \psi(t,x)=\frac{C_0}{2^{\beta-\sigma}\hat{C}}x_d^\beta.
\end{equation*}
By considering $-u$ instead of $u$, we have $u\ge-\frac{C_0}{2^{\beta-\sigma}\hat{C}}x_d^\beta$. Hence, around the origin $|u|\le C|x|^\beta$.  By rotation of the coordinate, we obtain the estimate near the lateral boundary.

For the bottom, let $\tilde \phi=C_0(t+1)$ so that $\tilde{\phi}(-1)=0$ and $\tilde{\phi}'(t)=C_0$. This yields that
  \begin{equation*}
  \partial_t\tilde{\phi}-\cM^+\phi=C_0.
  \end{equation*}
  Moreover, $\tilde{\phi}\ge 0$ in $(-1,0)\times\bR^d$. By the comparison principle again, $u\le \phi$ in $Q_1$. In particular, near the bottom $u\le C_0(t+1)$, which further implies $|u|\le C_0(t+1)$ by symmetry.

Combining the estimates of lateral boundary and bottom with the interior H\"older estimate, we prove the lemma.
\end{proof}

\section{A Liouville theorem}
The aim of this section is to prove the following Liouville theorem for the fully nonlinear parabolic nonlocal equation with non-symmetric kernels.  The elliptic version for symmetric kernels  was established in \cite{Serra15}.

\begin{theorem}\label{thm 10.251}
Let $\sigma\in (0,2)$. There is a constant $\hat{\alpha}\in (0,1/2)$ depending on $d$, $\lambda$, $\Lambda$, and $\sigma$ (but is uniform as $\sigma\to 2$) such that the following statement holds.
Let $\alpha\in (0,\hat{\alpha})$ be such that $[\sigma+\hat{\alpha}]<\sigma+\alpha$ and suppose that $u\in C_{\text{loc}}^{1+\frac{\alpha}{\sigma},\sigma+\alpha}(\bR_0^{d+1})$ satisfies the following properties:

(i) For any $\beta\in [0,\sigma+\alpha]$ and $R\ge 1$, we have
\begin{equation}
                                        \label{eq1.35}
[u]_{\beta/\sigma,\beta;Q_R}\le N_0 R^{\sigma+\alpha-\beta};
\end{equation}

(ii) For any $(s,h)\in \bR^{d+1}_0$, we have
\begin{align}\label{eq 10.131}
\partial_t\big(u(\cdot+s,\cdot+h)-u\big)
-\cM^-\big(u(\cdot+s,\cdot+h)-u\big)\ge 0,\\
\partial_t\big(u(\cdot+s,\cdot+h)-u\big)
-\cM^+\big(u(\cdot+s,\cdot+h)-u\big)\le 0;\label{eq 10.132}
\end{align}

(iii) If $\sigma>1$, for any nonnegative measure $\mu$ in $\bR^d$ with compact support,
\begin{align*}
\partial_t u_\mu-\mathcal{M}^+u_\mu\le 0,\quad\text{where}\quad u_\mu(t,x) = \int_{\bR^d}\delta u(t,x,h)\,d\mu(h).
\end{align*}
Then $u$ is a polynomial of degree $\nu$ in $x$ and $1$ in $t$, where $\nu$ is the integer part of $\sigma+\alpha$.
\end{theorem}

\begin{remark}
As in \cite{Serra15}, it is possible to relax Condition $(i)$ in Theorem \ref{thm 10.251} by assuming that \eqref{eq1.35} is satisfied for any $\beta\in [0,\sigma+\alpha']$ and $R\ge 1$, where $\alpha'\in (0,\alpha)$ is a constant satisfying $\sigma+\alpha'>\nu$. A simple computation reveals that in the case we also require that $\sigma>1+\alpha-\alpha'$ when $\sigma>1$ and $\nu=1$;
and $\sigma>\alpha-\alpha'$ when $\sigma<1$ and $\nu=0$.
\end{remark}

To prove Theorem \ref{thm 10.251}, we first present a few lemmas. Define
\begin{align*}
P(t,x)=\int_{\bR^d} \big(\delta u(t,x,y)-\delta u(0,0,y)\big)^+\frac{2-\sigma}{|y|^{d+\sigma}}\,dy,\\
N(t,x)=\int_{\bR^d} \big(\delta u(t,x,y)-\delta u(0,0,y)\big)^-\frac{2-\sigma}{|y|^{d+\sigma}}\,dy.
\end{align*}

\begin{lemma}
                        \label{lem3.3}
Let $\hat\alpha\in (0,1/2)$ be a constant satisfying $\hat\alpha<\sigma/2$. Under the conditions $(i)$ and $(ii)$ of Theorem \ref{thm 10.251}, for any $\kappa\ge 2$ and $l\in \mathbb{N}\cup \{0\}$, we have
\begin{equation}\label{eq 8.311a}
\sup_{Q_{\kappa^l}}\big(P+N+|u_t-u_t(0,0)|\big)\le CN_0\kappa^{\alpha l}\le CN_0\kappa^{\hat{\alpha}l}
\end{equation}
and
\begin{equation}\label{eq 10.142}
[u_t]_{\gamma/\sigma,\gamma;Q_{1/2}}\le CN_0\kappa^{\hat{\alpha}},
\end{equation}
where $C$ depends only on $d,\lambda,\Lambda$, and $\sigma$, and is uniformly bounded as $\sigma\to 2$.
\end{lemma}
\begin{proof}
We first estimate $P$ and $N$ assuming that $\nu=2$. Fix $(t,x), (t^\prime,x^\prime)\in Q_1$ and set $l=|x-x^\prime|+|t-t^\prime|^{1/\sigma}$. By Condition $(i)$, when $|y|<l$,
\begin{align}\nonumber
&|\delta u(t,x,y)-\delta u(t^\prime,x^\prime,y)|\\\nonumber
&=\Big|\int_{0}^1y\Big[Du(t,x+sy)-Du(t,x)-\big(Du(t^\prime,x^\prime+sy)
-Du(t^\prime,x^\prime)\big)\Big]\,ds\Big|\\
&\le C|y|^2l^{\sigma+\alpha-2}[u]_{1+\alpha/\sigma,\sigma+\alpha;Q_2}\le CN_0|y|^2l^{\sigma+\alpha-2}.\label{eq 8.86}
\end{align}
Similarly, when $|y|\ge l$,
\begin{align}\label{eq 11.032}
  |\delta u(t,x,y)-\delta u(t^\prime,x^\prime,y)|\le C|y|^{\sigma+\alpha-1}l[u]_{1+\alpha/\sigma,\alpha+\sigma;Q_{1+|y|}}\le CN_0l|y|^{\sigma+\alpha-1}.
\end{align}
Combining \eqref{eq 8.86} and \eqref{eq 11.032}, we have
 \begin{align}\nonumber
& \int_{\bR^d}|\delta u(t,x,y)-\delta u(t^\prime,x^\prime,y)|\frac{2-\sigma}{|y|^{d+\sigma}}\,dy\\\nonumber
 &\le CN_0l^{\sigma+\alpha-2}\int_{B_l}\frac{(2-\sigma)|y|^2}{|y|^{d+\sigma}}\,dy
+CN_0l\int_{\bR^d\setminus B_l}\frac{(2-\sigma)|y|^{\sigma+\alpha-1}}{|y|^{d+\sigma}}\,dy\\
 \label{eq 1104.2}
& \le CN_0l^{\alpha}.
 \end{align}
Hence $P,N\in C^{\alpha,\alpha/\sigma}(Q_1)$. Because $P(0,0)=N(0,0)=0$, we have
$$
P(t,x)+N(t,x)\le CN_0\quad \text{in}\,\, Q_1.
$$
By modifying the estimate above, we can prove the same estimate for $P$ when $\nu=0$ or $1$.

We then use a scaling argument. Define $\hat{u}(t,x)=\eta^{-\alpha-\sigma}u(\eta^\sigma t,\eta x)$ for any $\eta>1$. It is easily seen that $\hat{u}$ satisfies all the conditions in this lemma. Hence, we know that
 \begin{equation*}
 \int_{\bR^d}\frac{(2-\sigma)\big|\delta \hat{u}(t,x,y)-\delta \hat{u}(0,0,y)\big|}{|y|^{d+\sigma}}\,dy\le CN_0\quad \text{in}\,\, Q_1.
 \end{equation*}
Therefore,
$$
P(\eta^{\sigma}t,\eta x)+N(\eta^{\sigma}t,\eta x)\le CN_0\eta^{\alpha}\quad \text{in}\,\, Q_1,
$$
which together with \eqref{eq 10.131} and \eqref{eq 10.132} implies  \eqref{eq 8.311a}.

To prove \eqref{eq 10.142}, we take $h=0$ in \eqref{eq 10.131} and \eqref{eq 10.132}, and then multiply them by $1/s$. By letting $s\to 0$, we know that $u_t$ as well as $u_t-u_t(0,0)$ are  sub and super-solutions at the same time. By Proposition \ref{thm 10.131}, we obtain that $u_t\in C^{\gamma/\sigma,\gamma}(Q_{1/2})$ for some $\gamma>0$ depending on $d,\lambda,\Lambda$, and $\sigma_0$,  and
\begin{equation*}
[u_t]_{\gamma/\sigma,\gamma;Q_{1/2}}
=[u_t-u_t(0,0)]_{\gamma/\sigma,\gamma;Q_{1/2}}\le C\sup_{t\in (-1,0)}\int_{\bR^d}\frac{|u_t-u_t(0,0)|}{1+|x|^{\sigma+d}}\,dx.
\end{equation*}
Using \eqref{eq 8.311a}, for any $t\in (-1,0)$,
\begin{align*}
&\int_{\bR^d}\frac{|u_t-u_t(0,0)|}{1+|x|^{d+\sigma}}\,dx\,dt\\
&\le \int_{B_1}\frac{|u_t-u_t(0,0)|}{1+|x|^{d+\sigma}}\,dx\,dt
+\sum_{i=0}^\infty\int_{B_{\kappa^{i+1}}\setminus B_{\kappa^i}}\frac{|u_t-u_t(0,0)|}{1+|x|^{d+\sigma}}\,dx\,dt
 \end{align*}
\begin{align*}
 &\le CN_0\int_{B_1}\frac{1}{1+|x|^{d+\sigma}}\,dx\,dt+
CN_0\sum_{i=0}^\infty\int_{B_{\kappa^{i+1}}\setminus B_{\kappa^i}}\frac{\kappa^{\hat{\alpha}(i+1)}}{1+|x|^{d+\sigma}}\,dx\,dt\\
 &\le CN_0+CN_0\sum_{i=0}^\infty\kappa^{\hat{\alpha}(i+1)}
\int_{\kappa^{i}}^{\kappa^{i+1}}\frac{r^{d-1}}{1+r^{d+\sigma}}\,dr\\
&\le CN_0+CN_0 \frac{\kappa^{\hat\alpha}(1-\kappa^{-\sigma})}{1-\kappa^{\hat\alpha-\sigma}}
\le CN_0 \kappa^{\hat{\alpha}},
 \end{align*}
where $C$ depends only on $d,\lambda,\Lambda$, and $\sigma_0$. Here we used the fact $\hat{\alpha}<\sigma/2$ and $\kappa\ge 2$ in the last inequality.
Therefore, the lemma is proved.
\end{proof}

Dividing $u$ by a $CN_0$, where $C$ is the constant in \eqref{eq 8.311a}, and using Lemma \ref{lem3.3}, we have that for any $\kappa\ge 2$ and $l\in \mathbb{N}\cup \{0\}$,
\begin{equation}\label{eq 8.311}
\sup_{Q_{\kappa^l}}P\le \kappa^{\hat{\alpha}l},\quad
\sup_{Q_{\kappa^l}}N\le \kappa^{\hat{\alpha}l}.
\end{equation}
We are going to prove inductively that there exists a sufficiently large $\kappa\ge 2$ and sufficiently small $\hat\alpha\in (0,1/2)$ such that
\begin{equation*}
 \sup_{Q_{\kappa^{-l}}}P\le \kappa^{-\hat{\alpha} l},\quad \sup_{Q_{\kappa^{-l}}}N\le \kappa^{-\hat{\alpha} l}\quad \text{for any}\,\, l\in \mathbb{N}.
 \end{equation*}

For a fixed $r\in (0,1)$, assume that $P$ attains its maximum in $\overline{Q_r}$ at $(t_0,x_0)$. Denote
$$
\bA:=\big\{y: \delta u(t_0,x_0,y)-\delta u(0,0,y)>0\big\}.
$$
Then
\begin{align*}
P(t_0,x_0)&=\int_\bA\big(\delta u(t_0,x_0,y)-\delta u(0,0,y)\big)\frac{2-\sigma}{|y|^{d+\sigma}}\,dy,\\
N(t_0,x_0)&=\int_{\bR^d\setminus \bA}\big(\delta u(t_0,x_0,y)-\delta u(0,0,y)\big)\frac{2-\sigma}{|y|^{d+\sigma}}\,dy.
\end{align*}
We define
\begin{equation*}
v(t,x)=\int_\bA\big(\delta u(t,x,y)-\delta u(0,0,y)\big)\frac{2-\sigma}{|y|^{d+\sigma}}\,dy.
\end{equation*}
Notice that $v\le P$, and in particular $v\le 1$ in $Q_1$. Moreover, $P(t_0,x_0)=v(t_0,x_0)$.
We denote $\mathfrak{v}=(1-v)^+$.

\begin{lemma}
Suppose that $\hat\alpha\in (0,1/2)$ satisfying $\hat\alpha<\sigma/2$ and $\kappa\ge 2$.
Then we have
\begin{equation}\label{eq 10.1412}
\mathfrak{v}_t-\cM^-\mathfrak{v}\ge -C(\kappa^{\hat{\alpha}}-1)\quad \text{in}\,\, Q_{3/4},
\end{equation}
where $C$ is a positive constant depending only on $d$, $\lambda$, $\Lambda$, and $\sigma$, and is uniformly bounded as $\sigma\to 2$.
\end{lemma}
\begin{proof}
Since $v\le 1$ in $Q_1$, for any $(t,x)\in Q_1$ we have
$\mathfrak{v}(t,x)=1-v(t,x)$, thus
\begin{align*}
&\delta \mathfrak{v}(t,x,y)=\mathfrak{v}(t,x+y)-\mathfrak{v}(t,x)-D\mathfrak{v}(t,x)y\\
&=(1-v)^+(t,x+y)-(1-v)(t,x)+Dv(t,x)y\\
&=(v-1)^+(t,x+y)-\delta v(t,x,y).
\end{align*}
Therefore, we have
\begin{align*}
\big(\delta \mathfrak{v}(t,x,y)\big)^-\ge \big(\delta v(t,x,y)\big)^+-(v-1)^+(t,x+y),\\
\big(\delta \mathfrak{v}(t,x,y)\big)^+\le \big(\delta v(t,x,y)\big)^-+(v-1)^+(t,x+y).
\end{align*}
These imply
\begin{align}\nonumber
&\mathfrak{v}_t-\cM^-\mathfrak{v}=-v_t-(2-\sigma)\int_{\bR^d}\frac{\lambda (\delta \mathfrak{v}(t,x,y))^+-\Lambda(\delta \mathfrak{v}(t,x,y))^-}{|y|^{d+\sigma}}\,dy\\
&\ge -v_t+\cM^+v-(2-\sigma)(\Lambda+\lambda)\int_{\bR^d}\frac{(v-1)^+(t,x+y)}
{|y|^{d+\sigma}}\,dy.\label{eq 12.051}
\end{align}
From Condition $(iii)$ and an approximation (see, for instance, the proof of Lemma \ref{lem4.4} below), $v$ satisfies
\begin{equation}\label{eq 12.052}
v_t-\cM^+v\le 0.
\end{equation}
On the other hand, we have
\begin{align*}
\int_{\bR^d}\frac{(v-1)^+(t,x+y)}{|y|^{d+\sigma}}\,dy\le \int_{\bR^d}\frac{(P-1)^+(t,x+y)}{|y|^{d+\sigma}}\,dy.
\end{align*}
Since $P$ satisfies \eqref{eq 8.311}, for $(t,x)\in Q_{3/4}$, we have $P(t,x+y)\le 1$ when $y\in B_{1/4}$, and thus the right-hand side above is equal to
\begin{align*}
\sum_{i=0}^\infty\int_{B_{\kappa^{i+1}-3/4}\setminus B_{\kappa^i-3/4}}\frac{(P-1)^+(t,x+y)}{|y|^{d+\sigma}}\,dy,
\end{align*}
which by \eqref{eq 8.311} is bounded by
\begin{align}\nonumber
&\sum_{i=0}^\infty\int_{B_{\kappa^{i+1}-3/4}\setminus B_{\kappa^i-3/4}}\frac{\kappa^{(i+1)\hat{\alpha}}-1}{|y|^{d+\sigma}}\,dy\\
 &\le C \sum_{i=0}^\infty(\kappa^{(i+1)\hat{\alpha}}-1)\kappa^{-i\sigma}
=C\Big(\frac{\kappa^{\hat{\alpha}}}{1-\kappa^{\hat{\alpha}-\sigma}}
-\frac 1{1-\kappa^{-\sigma}}\Big)\le C\big(\kappa^{\hat{\alpha}}-1\big),\label{eq 12.054}
 \end{align}
where $C$ only depends on $d,\lambda,\Lambda$, and $\sigma_0$.  Here we used the fact 
$\hat{\alpha}<\sigma/2$ and $\kappa\ge 2$ in the last inequality.
Combining \eqref{eq 12.052}-\eqref{eq 12.054} with \eqref{eq 12.051}, we prove the lemma.

\end{proof}

Let $\hat{\theta}=\lambda/(4\Lambda)$ and $\gamma$ be the constant in Proposition \ref{thm 10.131}.
For any $r_1>0$, define the set
$$\cD_{r_1}=\{(t,x)\in {Q}_{r_1}: v\ge 1-\hat{\theta}\}.$$

\begin{lemma}
                    \label{lem3.5}
Suppose that $\hat\alpha\in (0,1/2)$ satisfying $\hat\alpha<\sigma/2$ and $\kappa\ge 2$. There exist some $\eta\in (0,1)$ sufficiently close to 1 and $c\in (0,1)$ sufficiently small, both depending only on $d$, $\lambda$, $\Lambda$, and $\sigma$ (and is uniform as $\sigma\to 2$), such that for $r_1= c\kappa^{-\hat{\alpha}/\gamma}$,
\begin{equation}
                                \label{eq 10.153}
|\cD_{r_1}|\le \eta|{Q}_{r_1}|.
\end{equation}
\end{lemma}

\begin{proof}
By contradiction we assume that $|\cD_{r_1}|> \eta|{Q}_{r_1}|$, and consider
\begin{equation*}
w:=\int_{\bR^d\setminus \bA}\big(\delta u(t,x,y)-\delta u(0,0,y)\big)\frac{2-\sigma}{|y|^{d+\sigma}}\,dy.
\end{equation*}
By Condition $(iii)$ and an approximation argument, $w$ is a subsolution, i.e.,
\begin{equation}
w_t-\cM^+w\le 0 \quad \text{in}\quad \bR^{d+1}_0.\label{eq 1221.3}
\end{equation}
From \eqref{eq 10.131} and \eqref{eq 10.132}, we know that in $\bR_0^{d+1}$,
\begin{equation}\label{eq 10.133}
\frac{\lambda}{\Lambda}P-\frac{u_t-u_t(0,0)}{\Lambda}\le N\le\frac{\Lambda}{\lambda}P-\frac{u_t-u_t(0,0)}{\lambda},
\end{equation}
implying that in $Q_{1/2}$
\begin{equation*}
\frac{\lambda}{\Lambda}P-\frac{[u_t]_{\gamma/\sigma,\gamma;Q_{1/2}}}{\Lambda}
\big(|x|^\gamma+|t|^{\gamma/\sigma}\big)
\le N\le \frac{\Lambda}{\lambda}P+\frac{[u_t]_{\gamma/\sigma,\gamma;Q_{1/2}}}
{\lambda}\big(|x|^\gamma+|t|^{\gamma/\sigma}\big).
\end{equation*}
From \eqref{eq 10.142}, \eqref{eq 8.311}, and the above inequality,
we obtain
\begin{align*}
w&=P-v-N\le \big(1-\lambda/\Lambda\big)P-v+C\kappa^{\hat{\alpha}}
\big(|x|^\gamma+|t|^{\gamma/\sigma}\big)\\
&\le 1-\lambda/\Lambda-(1-\hat{\theta})+C\kappa^{\hat{\alpha}}
\big(|x|^\gamma+|t|^{\gamma/\sigma}\big)\\
&\le -\lambda/\Lambda+\hat{\theta}+C\kappa^{\hat{\alpha}}
\big(|x|^\gamma+|t|^{\gamma/\sigma}\big)\quad \text{in}\,\, \cD_{r_1},
\end{align*}
where $C$ only depends on $d,\lambda,\Lambda$, and $\sigma$, and is uniformly bounded as $\sigma\to 2$.
Now we choose $c$ sufficiently small  depending only on $d,\lambda,\Lambda$, and $\sigma$ (uniformly as $\sigma\to 2$), such that
\begin{align}
&-\lambda/\Lambda+\hat{\theta}+C\kappa^{\hat{\alpha}}
\big(|x|^\gamma+|t|^{\gamma/\sigma}\big)\nonumber\\
&\le -3\hat{\theta}+C\kappa^{\hat{\alpha}}
\big(c\kappa^{-\hat{\alpha}/\gamma}\big)^\gamma\le-3\hat{\theta}+Cc^\gamma \le -\hat{\theta}\quad \text{in}\,\, Q_{r_1},\label{eq 1.041}
\end{align}
which implies $w\le -\hat{\theta}$ in $\cD_{r_1}$.
Since $w$ is a subsolution \eqref{eq 1221.3}, it follows immediately that for any $\epsilon\in(0,r_1)$, $\bar{w}_\epsilon(t,x):=(w+\hat{\theta})^+(\epsilon^\sigma t,\epsilon x)$ is a subsolution as well. Moreover,
\begin{equation}\label{eq 10.146}
\Big|\{\bar{w}_{\epsilon}\le 0\}\cap Q_{r_1/\epsilon}\Big|
>\eta\Big|Q_{r_1/\epsilon}\Big|.
\end{equation}

We estimate $\bar{w}_\epsilon$ by applying Proposition \ref{thm 8.261} with $t_1=-1$, $t_2=0$, and $\Omega=\bR^d$
\begin{align}
\nonumber
&\bar{w}_\epsilon(0,0)\le C\int_{-1}^0\int_{\bR^d}\frac{\bar{w}_\epsilon}{1+|x|^{d+\sigma}}\,dx\,dt\\
&=C\int_{-1}^0\int_{B_{r_1/\epsilon}}\frac{\bar{w}_\epsilon}{1+|x|^{d+\sigma}}\,dx\,dt
+C\int_{-1}^0\int_{B_{r_1/\epsilon}^c}\frac{\bar{w}_\epsilon}{1+|x|^{d+\sigma}}\,dx\,dt.\label{eq 10.101}
\end{align}
We first consider the second term on the right-hand side of the inequality above
\begin{align}\nonumber
&\int_{-1}^0\int_{B^c_{r_1/\epsilon}}\frac{\bar{w}_\epsilon}{1+|x|^{d+\sigma}}\,dx\,dt\\
\nonumber
&\le \int_{-1}^0\int_{B^c_{r_1/\epsilon}}\frac{\hat{\theta}}{1+|x|^{d+\sigma}}\,dx\,dt
+\int_{-1}^0\int_{B^c_{r_1/\epsilon}}\frac{|w|(\epsilon^\sigma t,\epsilon x)}{1+|x|^{d+\sigma}}\,dx\,dt\\ \label{eq 10.143}
&\le C\hat{\theta}(\epsilon/r_1)^{\sigma}+\int_{-\epsilon^\sigma}^0\int_{B_{r_1}^c}\frac{|w|(t,x)}{\epsilon^{d+\sigma}+|x|^{d+\sigma}}\,dx\,dt.
\end{align}
Since $|w|\le \max\{P,N\}$, from \eqref{eq 8.311}, for any $l\ge 0$,
$$\sup_{Q_{\kappa^l}}|w|\le \kappa^{\hat{\alpha} l}.$$
Therefore,
\begin{align}\nonumber
&\int_{-\epsilon^\sigma}^0\int_{B_{r_1}^c}\frac{|w|(t,x)}{\epsilon^{d+\sigma}+|x|^{d+\sigma}}\,dx\,dt\\ \nonumber
&\le \int_{-\epsilon^{\sigma}}^0\int_{B_1\setminus B_{r_1}}\frac{|w|}{\epsilon^{d+\sigma}+|x|^{d+\sigma}}\,dx\,dt+\sum_{i=0}^\infty\int_{-\epsilon^\sigma}^0\int_{B_{\kappa^{i+1}}\setminus B_{\kappa^i}}\frac{|w|}{|x|^{d+\sigma}}\,dx\,dt\\
\label{eq 10.144}
&\le C\big((\epsilon/r_1)^\sigma+\epsilon^\sigma\kappa^{\hat{\alpha}}\big),
\end{align}
where $C$ only depends on $d,\lambda,\Lambda$, and $\sigma$, and is uniformly bounded as $\sigma\to 2$.
We combine \eqref{eq 10.143} and \eqref{eq 10.144} to obtain that
\begin{align*}
C\int_{-1}^0\int_{B^c_{r_1/\epsilon}}\frac{\bar{w}_\epsilon}{1+|x|^{d+\sigma}}\,dx\,dt\le C\big((\epsilon/r_1)^\sigma+\epsilon^\sigma \kappa^{\hat{\alpha}}\big).
\end{align*}
Recall that $r_1= c\kappa^{-\hat{\alpha}/\gamma}$. For the right-hand side of the inequality above, we want to choose $\epsilon$ sufficiently small such that
$C\epsilon^\sigma(r_1^{-\sigma}+\kappa^{\hat{\alpha}})\le \hat{\theta}/4$.
Indeed, since  $c\in (0,1)$  and $\sigma>\gamma$, we have
$r_1^{-\sigma}\ge \kappa^{\hat{\alpha}}$.
It is sufficient to fix $\epsilon$ such that $2C\epsilon^\sigma r_1^{-\sigma}= \hat{\theta}/4$,
i.e.,
$$
\epsilon/r_1= \big(\hat{\theta}/(8C)\big)^{1/\sigma}:=c_1,
$$
where $c_1$ only depends on $d,\lambda,\Lambda$, and $\sigma$, and is uniformly bounded as $\sigma\to 2$.
In other words, by taking $\epsilon=c_1r_1$ we have
\begin{equation}\label{eq 10.148}
 C\int_{-1}^0\int_{B^c_{r_1/\epsilon}}\frac{\bar{w}_\epsilon}{1+|x|^{d+\sigma}}\,dx\,dt\le \hat{\theta}/4.
\end{equation}

Next we estimate the first term on the right-hand side of \eqref{eq 10.101} using \eqref{eq 10.146}:
\begin{align}
&C\int_{-1}^0\int_{B_{1/c_1}}\frac{\bar{w}_\epsilon}{1+|x|^{d+\sigma}}
\,dx\,dt\nonumber\\
&=C\int_{\big((-1,0)\times B_{1/c_1}\big)\cap \{\bar{w}_\epsilon>0\}}\frac{\bar{w}_\epsilon}{1+|x|^{d+\sigma}}\,dx\,dt\nonumber\\
&\le C(1-\eta)|Q_{1/c_1}|\sup_{(-1,0)\times B_{1/c_1}}\bar{w}_\epsilon\nonumber\\
&\le C(1-\eta)|Q_{1/c_1}|\sup_{(-1,0)\times B_{r_1}}(w+\hat\theta)^+
\le \hat{\theta}/4
\label{eq 10.147}
\end{align}
upon taking $\eta$ sufficiently close to $1$ depending only on $d,\lambda,\Lambda$, and $\sigma$ (uniformly as $\sigma\to 2$).

Combining  \eqref{eq 10.147} and \eqref{eq 10.148}  with \eqref{eq 10.101}, we have $\bar{w}_\epsilon(0,0)\le \hat{\theta}/2$ indicating that $w(0,0)\le -\hat{\theta}/2$, which contradicts with $w(0,0)=0$ by the definition of $w$.
Therefore, the lemma is proved.
\end{proof}

Now we are ready to prove Theorem \ref{thm 10.251}.

\begin{proof}[Proof of Theorem  \ref{thm 10.251}]
Below we first elaborate on the case when $\sigma>1$.  At the end, we briefly discuss the case when $\sigma=1$. The proof of the case when $\sigma<1$ is omitted due to similarity to the first case.

Let $\eta$ and $c$ be the constants in Lemma \ref{lem3.5} and $r_1= c\kappa^{-\hat{\alpha}/\gamma}$. From \eqref{eq 10.153},
$$
\big|\{1-v\ge \hat{\theta}\}\cap Q_{r_1}\big|
\ge (1-\eta)\big|Q_{r_1}\big|.
$$
Recall that $\tilde{Q}_{\delta r_1}=(-{r_1}^\sigma,-(\delta {r_1})^\sigma)\times B_{r_1}$. For $\delta$ sufficiently small depending on $\eta$, we have
$$
\big|\{1-v\ge\hat{\theta}\}\cap \tilde{Q}_{\delta r_1}\big|
\ge \frac{1-\eta}{2} \big|\tilde{Q}_{\delta r_1}\big|.
$$
We set $r=\delta r_1/2$ and apply the weak Harnack inequality Corollary \ref{weak_harnack} to $(1-v)^+$ in $Q_{3/4}$ with \eqref{eq 10.1412} to obtain
\begin{align}
&\inf_{Q_{r}}(1-v)+C(\kappa^{\hat\alpha}-1)r_1^{\sigma}\nonumber\\
&\ge C(\delta)\|(1-v)^+\|_{L_{\epsilon_1}(\tilde{Q}_{\delta r_1})}r_1^{-(d+\sigma)/\epsilon_1}\nonumber\\
&\ge C(\delta)\hat{\theta}\big((1-\eta)|\tilde{Q}_{\delta r_1}|/2\big)^{1/\epsilon_1}r_1^{-(d+\sigma)/\epsilon_1}\nonumber\\
&=C(\delta)\hat{\theta}
\big((1-\eta)(1-\delta^\sigma)/2\big)^{1/\epsilon_1}=:2\theta,    \label{eq4.40}
\end{align}
where $\theta$ is a small constant depending only on $d$, $\lambda$, $\Lambda$, and $\sigma$, and is uniformly bounded as $\sigma\to 2$.
We fix $\kappa=4/(c\delta)^2$, where $c$ is chosen according to \eqref{eq 1.041},  which guarantees that
\begin{equation*}
\kappa^{-1}\le c\delta/2\cdot\kappa^{-\hat{\alpha}/\gamma}
\end{equation*}
for any $\hat{\alpha}\in (0,\gamma/2)$.
Since $r = \delta r_1/2$ and $r_1 = c\kappa^{-\hat{\alpha}/\gamma}$, the inequality above  can be written as $\kappa^{-1}<r$.

Next, we choose $\hat{\alpha}_1=\log(1+\theta/C)/\log\kappa$ such that for any $\hat{\alpha}\in (0, \hat{\alpha}_1)$,
\begin{equation*}
C(\kappa^{\hat\alpha}-1)r_1^\sigma
<C(\kappa^{\hat\alpha}-1)<\theta.
\end{equation*}
Therefore, from \eqref{eq4.40} and the inequality above we obtain that $\sup_{Q_{r}}v\le 1-\theta$.
Since $\kappa^{-1}<r$,
\begin{equation}\label{eq 10.121}
\sup_{Q_{\kappa^{-1}}}P\le \sup_{Q_{r}}P
=P(t_0,x_0)=v(t_0,x_0)=\sup_{Q_{r}}v\le
1-\theta.
\end{equation}
Similarly,
\begin{equation}\label{eq 10.121b}
\sup_{Q_{\kappa^{-1}}}N \le
1-\theta.
\end{equation}

Set
$$
\hat{\alpha}=\min\big\{-\log(1-\theta)/\log \kappa,\hat{\alpha}_1,\gamma/2\big\}.
$$
Then \eqref{eq 10.121} and \eqref{eq 10.121b} imply
\begin{equation}
                                \label{eq9.04}
\sup_{Q_{\kappa^{-1}}}\kappa^{\hat{\alpha}}P\le 1,\quad
\sup_{Q_{\kappa^{-1}}}\kappa^{\hat{\alpha}}N\le 1.
\end{equation}
Let
$$
\tilde{P}(t,x)=\kappa^{\hat{\alpha}}P(\kappa^{-1}x,\kappa^{-\sigma}t),
\quad \tilde{N}(t,x)=\kappa^{\hat{\alpha}}N(\kappa^{-1}x,\kappa^{-\sigma}t),
$$
and
$$
\tilde u(t,x)=\kappa^{\sigma+\hat{\alpha}}u(\kappa^{-1}x,\kappa^{-\sigma}t).
$$
From \eqref{eq 10.133}, we have
\begin{equation*}
\frac{\lambda}{\Lambda}\tilde{P}-
\frac{\tilde u_t-\tilde u_t(0,0))}{\Lambda}
\le \tilde{N}
\le\frac{\Lambda}{\lambda}\tilde{P}-
\frac{\tilde u_t-\tilde u_t(0,0)}{\lambda}.
\end{equation*}
Since  $\kappa\ge 2$ and $\hat{\alpha}\le\gamma$, we get
\begin{align*}
[\tilde u_t]_{\gamma/\sigma,\gamma;Q_{1/2}}\le
\kappa^{\hat\alpha-\gamma}[u_t]_{\gamma/\sigma,\gamma;Q_{1/2}}\le
[u_t]_{\gamma/\sigma,\gamma;Q_{1/2}}.
\end{align*}
On the other hand, for any $l\ge 0$,
\begin{equation*}
\sup_{Q_{\kappa^l}}\tilde{P}=\kappa^{\hat{\alpha}}\sup_{Q_{\kappa^{l-1}}}P\le \kappa^{\hat{\alpha}}\kappa^{(l-1)\hat{\alpha}}=\kappa^{l\hat{\alpha}},
\quad \sup_{Q_{\kappa^l}}\tilde{N}\le \kappa^{l\hat{\alpha}}.
\end{equation*}
Therefore, $\tilde{P}$ and $\tilde N$ satisfy all the conditions that $P$ and $N$ satisfied.
Applying \eqref{eq9.04} to $\tilde{P}$ and $\tilde N$, we have
$$
\sup_{Q_{\kappa^{-1}}}\tilde{P}\le \kappa^{-\hat{\alpha}},\quad
\sup_{Q_{\kappa^{-1}}}\tilde{N}\le \kappa^{-\hat{\alpha}},
$$
which further implies that
\begin{equation*}
\sup_{Q_{\kappa^{-2}}}P\le \kappa^{-2\hat{\alpha}},\quad
\sup_{Q_{\kappa^{-2}}}N\le \kappa^{-2\hat{\alpha}}.
\end{equation*}
By induction, for any $l\in \mathbb{N}$,
\begin{equation*}
\sup_{Q_{\kappa^{-l}}}P\le \kappa^{-l\hat{\alpha}},\quad
\sup_{Q_{\kappa^{-l}}}N\le \kappa^{-l\hat{\alpha}}.
\end{equation*}
Therefore, we have in $Q_1$
\begin{equation*}
P(t,x)\le C\big(|x|^{\hat{\alpha}}+|t|^{\hat{\alpha}/\sigma}\big),\quad
N(t,x)\le C\big(|x|^{\hat{\alpha}}+|t|^{\hat{\alpha}/\sigma}\big).
\end{equation*}

Since for any $\eta\ge 1$, $\hat{u}(t,x)=\eta^{-\sigma-\alpha}u(\eta^{\sigma}t,\eta x)$ satisfies the same condition as $u$,  replacing $u$ by $\hat{u}$ in the definition of $P$ and denoting it as $P_{\hat{u}}$, we obtain
\begin{equation*}
P_{\hat{u}}(t,x)\le C\big(|x|^{\hat{\alpha}}+|t|^{\hat{\alpha}/\sigma}\big)\quad \text{in}\,\, Q_1.
\end{equation*}
Returning to  $P$, we have
\begin{equation*}
\eta^{-\alpha}P(\eta^\sigma t,\eta x)\le C\big(|x|^{\hat{\alpha}}+|t|^{\hat{\alpha}/\sigma}\big)
\end{equation*}
in $Q_1$, which further implies that
\begin{equation*}
\sup_{Q_\eta}\frac{P(t,x)}{|x|^{\hat{\alpha}}+|t|^{\hat{\alpha}/\sigma}}\le C\eta^{\alpha-\hat{\alpha}}.
\end{equation*}
Let $\eta\to \infty$ yields
\begin{equation*}
\sup_{(t,x)\in \bR^{d+1}_0}\frac{P(t,x)}{|x|^{\hat{\alpha}}+|t|^{\hat{\alpha}/\sigma}}=0,
\end{equation*}
which gives $P=0$. Similarly, $N=0$.

From the definition of $P$ and $N$, we have
$$u(t,x+y)-u(t,x)-Du(t,x)y=u(0,y)-u(0,0)-Du(0,0)y.$$
Taking derivative in $y$, we have, for any $t\in (-\infty,0)$ and $x,y\in \bR^d$
\begin{equation*}
Du(t,x+y)-Du(t,x)=Du(0,y)-Du(0,0),
\end{equation*}
which implies for fixed $t$, $u$ is a polynomial in $x$ of order at most two. Using Condition $(i)$ with $\beta=0$, we infer that this order is at most $\nu$.  Condition $(ii)$, together with  $P=N=0$, yields $u_t=c$ for some constant $c$. The proof is completed for $\sigma>1$.

Finally, we sketch the proof for $\sigma=1$. From Condition $(ii)$,  we know that $u(\cdot+s,\cdot+h)-u$, and thus $Du$ and $u_t$ are both sub and supersolutions, and are in $C^{\gamma/\sigma,\gamma}(Q_{1/2})$. 
By Proposition \ref{thm 10.131}, we have
\begin{equation*}
[u_t]_{\gamma/\sigma,\gamma;Q_{1/2}}+[Du]_{\gamma/\sigma,\gamma;Q_{1/2}}\le C\sup_{t\in[0,1]}\int_{\bR^d}\frac{|Du|+|u_t|}{1+|x|^{d+1}}\,dx.
\end{equation*}
From Condition $(i)$, the right-hand side of the inequality above is less than
\begin{equation*}
\int_{B_1}\frac{1}{1+|x|^{1+d}}\,dx+\sum_{j=1}^\infty\int_{B_{2^j}\setminus B_{2^{j-1}}}\frac{C2^{j\alpha}}{1+|x|^{1+d}}\,dx\le \frac{C}{1-2^{\alpha-1}}
\end{equation*}
for any $\alpha\in (0,1)$. By taking $\hat{\alpha}\le \gamma$ and scaling as before, we can prove that
\begin{equation*}
[u]_{1+\gamma,1+\gamma;Q_R}\le CR^{\alpha-\gamma},
\end{equation*}
i.e., $u$ must be a linear function by sending $R\to \infty$. The theorem is proved.
\end{proof}

\section{Schauder estimate for nonlocal parabolic equations}
In this section, we prove Theorem \ref{thm 1} by applying the Liouville theorem, a blow-up analysis, and a localization procedure.
In the rest of the paper, we do not specify the domain associated with the norm when it is $\bR^{d+1}_0=(-\infty,0)\times \bR^d$.

\subsection{Equations with translation invariant kernels}
In this subsection, we consider equations with translation invariant kernels, i.e., $K=K(y)$. The main result of the subsection is the following theorem.

\begin{theorem}\label{thm 1104.1}
Let $\sigma\in(0,2)$ be a constant and $\cA$ be an index set. There exists a constant $\hat{\alpha}>0$ depending on $d$, $\lambda$, $\Lambda$, and $\sigma$ (uniformly as $\sigma\to 2$), such that given $0<\alpha'<\alpha<\hat{\alpha}$ satisfying $[\sigma+\alpha]<\sigma+\alpha'<\sigma+\alpha$ the following holds. Let $u\in C^{1+\alpha'/\sigma,\alpha'+\sigma}\big((-1,0)\times\bR^d\big)\cap C^{1+\alpha/\sigma,\alpha+\sigma}(Q_1)$ satisfy
\begin{equation*}
 u_t=\inf_{a\in \mathcal{A}}\big({L_a u+f_a}\big)\quad \text{in}\,\, Q_1,
 \end{equation*}
 where $L_a\in \mathcal{L}_{0}(\sigma,\lambda,\Lambda)$ with $K_a=K_a(y)$ for any $a\in \mathcal{A}$. Assume that
$$
\sup_{(t,x)\in Q_1}\Big|\inf_{a\in \mathcal{A}}f_a(t,x)\Big|<\infty.
$$
Then
\begin{equation*}
 [u]_{1+\alpha/\sigma,\alpha+\sigma;Q_{1/2}}\le C[u]_{1+\alpha'/\sigma,\alpha'+\sigma;(-1,0)\times\bR^d}+C\sup_{a\in \mathcal{A}}[f_a]_{\alpha/\sigma,\alpha;Q_1},
\end{equation*}
where $C$ only depends on $d,\lambda,\Lambda,\sigma,\alpha$, and $\alpha'$, and is uniformly bounded as $\sigma\to 2$.
\end{theorem}

We denote
\begin{equation}
			\label{eq 3.11}
Q^k=\big(-1+2^{-(k+1)\sigma}/(1-2^{-\sigma}),0\big)\times B_{1-2^{-k}}(0)
\end{equation}
for all sufficiently large integers $k$ such that $2^{-(k+1)\sigma}<1-2^{-\sigma}$. We shall prove a stronger result:
\begin{align}\label{eq 1.151}
&\sup_k2^{-k(\alpha-\alpha')}[u]_{1+\alpha/\sigma,\alpha+\sigma;Q^k}\nonumber\\
&\le C[u]_{1+\alpha'/\sigma,\alpha'+\sigma;(-1,0)\times\bR^d}
+C\sup_{a\in \mathcal{A}}[f_a]_{\alpha/\sigma,\alpha;Q_1}.
\end{align}
The conclusion of the theorem is a particular case for $k$  large only depending on $\sigma$ (uniformly as $\sigma\to 2$) so that $Q_{1/2}\subset Q^k$.  Since we assume that $u\in C^{1+\alpha/\sigma,\alpha+\sigma}(Q_1)$, there exists an integer $k$ such that
 \begin{equation*}
2^{-k(\alpha-\alpha')}
[u]_{1+\alpha/\sigma,\alpha+\sigma;Q^k}
=\sup_l2^{-l(\alpha-\alpha')}[u]_{1+\alpha/\sigma,\alpha+\sigma;Q^l}.
\end{equation*}

Next, we prove \eqref{eq 1.151} by contradiction.
Assume that we can find solutions $u_j$ and index sets $\mathcal{A}_j$ such that
\begin{align}
&\partial_tu_j=\inf_{a\in \mathcal{A}_j}(L_au_j+f_a)\quad \text{in}\quad Q_1,\quad \sup_{(t,x)\in Q_1}\Big|\inf_{a\in \mathcal{A}_j}f_a(t,x)\Big|<\infty,\nonumber\\
&[u_j]_{1+\alpha^\prime/\sigma,\sigma+\alpha^\prime;(-1,0)\times\bR^d}+\sup_{a\in \mathcal{A}_j}[f_a]_{\alpha/\sigma,\alpha;Q_1}\le 1,\nonumber\\
                            \label{eq2.26}
&\text{and}\quad \sup_k2^{-k(\alpha-\alpha')}[u_j]_{1+\alpha/\sigma,\sigma+\alpha;Q^k}\ge j,
\end{align}
where for any $a\in \mathcal{A}_k$, $L_a\in \mathcal{L}_0$ with $K_a=K_a(y)$. As explained above, for each $j$ there exists an integer $k_j$ so that
\begin{equation*}
2^{-k_j(\alpha-\alpha')}[u_j]_{1+\alpha/\sigma,\alpha+\sigma;Q^{k_j}}=\sup_k2^{-k(\alpha-\alpha')}
[u_j]_{1+\alpha/\sigma,\alpha+\sigma;Q^k}.
\end{equation*}

\begin{lemma}
                                \label{lem4.2}
For any $j\ge 1$, we have
\begin{align}
[u_j]_{1+\alpha/\sigma,\alpha+\sigma;Q^{k_j}}&\le \sup_{r>0}\sup_{(t,x)\in Q^{k_j}}r^{-(\alpha-\alpha')}
[u_j]_{1+\alpha'/\sigma,\alpha'+\sigma;Q_r(t,x)\cap \{t>-1\}}\nonumber\\
                                \label{eq10.16}
&\le 4^{\alpha-\alpha'}[u_j]_{1+\alpha/\sigma,\alpha+\sigma;Q^{k_j}}.
\end{align}
Moreover, we can find $(t_j,x_j)\in Q^{k_j}$ and $r_j$ such that
\begin{equation}
\frac{1}{2}[u_j]_{1+\alpha/\sigma,\alpha+\sigma;Q^{k_j}}\le r_j^{-(\alpha-\alpha')}[u_j]_{1+\alpha'/\sigma,\alpha'+\sigma;Q_{r_j}(t_j,x_j)\cap \{t>-1\}}
\label{eq 1.161}
\end{equation}
and
\begin{equation}
2^{k_j}r_j\to 0\quad \text{as}\,\,j\to \infty.
                                \label{eq 1.162}
\end{equation}
\end{lemma}
\begin{proof}
The first inequality in \eqref{eq10.16} follows from the fact that for any $(t,x),(s,y)\in Q^{k_j}$ with $t\ge s$, we have $(s,y)\in Q_r(t,x)\cap \{t>-1\}$, where $r=\max(|x-y|,|t-s|^{1/\sigma})$. See, for instance, Claim 3.2 of \cite{Serra15}. For the second inequality, if $r\le 2^{-(k_j+1)}$, for any $(t,x)\in Q^{k_j}$, we have $Q_r(t,x)\subset Q^{k_j+1}$ and
\begin{align*}
&r^{-(\alpha-\alpha')}
[u_j]_{1+\alpha'/\sigma,\alpha'+\sigma;Q_r(t,x)\cap \{t>-1\}}\\
&\le 2^{\alpha-\alpha'}
[u_j]_{1+\alpha/\sigma,\alpha+\sigma;Q^{k_j+1}}\le  4^{\alpha-\alpha'}[u_j]_{1+\alpha/\sigma,\alpha+\sigma;Q^{k_j}},
\end{align*}
where the last inequality is due to the choice of $k_j$. On the other hand, if $r> 2^{-(k_j+1)}$, for any $(t,x)\in Q^{k_j}$,
\begin{align*}
&r^{-(\alpha-\alpha')}
[u_j]_{1+\alpha'/\sigma,\alpha'+\sigma;Q_r(t,x)\cap \{t>-1\}}\\
&\le 2^{(k_j+1)(\alpha-\alpha')}
[u_j]_{1+\alpha'/\sigma,\alpha'+\sigma;(-1,0)\times \bR^d}\le  2^{\alpha-\alpha'}[u_j]_{1+\alpha/\sigma,\alpha+\sigma;Q^{k_j}},
\end{align*}
where the last inequality follows from \eqref{eq2.26}. Thus, we obtain the second inequality in \eqref{eq10.16}.

Due to \eqref{eq10.16}, we can find $(t_j,x_j)\in Q^{k_j}$ and $r_j$ such that \eqref{eq 1.161} is satisfied
and thus by \eqref{eq2.26},
$$
(2^{k_j}r_j)^{\alpha-\alpha'}\le \frac{2[u_j]_{1+\alpha'/\sigma,\alpha'+\sigma;Q_{r_j}(t_j,x_j)\cap \{t>-1\}}}
{2^{-k_j(\alpha-\alpha')}
[u_j]_{1+\alpha/\sigma,\alpha+\sigma;Q^{k_j}}}\to 0 \quad \text{as}\,\, j\to \infty,
$$
which further implies \eqref{eq 1.162}. The lemma is proved.
\end{proof}

Let $T_j$ be the Taylor expansion of $u_j$ at $X_j=(t_j,x_j)$ of order $\nu=[\sigma+\alpha]$ in $x$ and 1 in $t$.
 Now we consider the blow-up sequence
\begin{equation*}
v_j(t,x)=\frac{u_j(t_j+r_j^\sigma t,x_j+r_jx)-T_j(t_j+r_j^\sigma t,x_j+r_jx)}{r_j^{\sigma+\alpha}[u_j]_{1+\alpha/\sigma,\alpha+\sigma;Q^{k_j}}}.
\end{equation*}

Here $(t_j,x_j)$ and $r_j$ are from Lemma \ref{lem4.2}. Note that $v_j$ is well defined on $(-R_j^\sigma,0)\times\bR^d$, where by Lemma \ref{lem4.2},
$$
R_j:=2^{-(k_j+1)}r_j^{-1}\to \infty\quad\text{as}\quad
j\to \infty.
$$
Observe that from \eqref{eq 1.161} and \eqref{eq 1.162}, for sufficiently large $j$ such that $r_j\le 2^{-(k_j+1)}$,
\begin{align}\label{eq 1.191}
[v_j]_{1+\alpha'/\sigma,\alpha'+\sigma;Q_1}
=\frac{r_j^{\sigma+\alpha'}[u_j]_{1+\alpha'/\sigma,\alpha'+\sigma;Q_{r_j}(t_j,x_j)}}
{r_j^{\sigma+\alpha}[u_j]_{1+\alpha/\sigma,\alpha+\sigma;Q^{k_j}}}\ge 1/2.
\end{align}

\begin{lemma}
                                \label{lem4.3}
For any $R>0$ and $\beta\in [0,\sigma+\alpha^\prime]$, we have
\begin{equation}
                                \label{eq 10.211}
[v_j]_{\beta/\sigma,\beta;Q_R \cap\{t>-R_j^\sigma\}}\le CR^{\sigma+\alpha-\beta},
\end{equation}
where $C$ depends only on $\alpha$ and $\alpha'$. Moreover, for any $0<R<R_j$ and $\beta\in [0,\sigma+\alpha]$, we have
\begin{equation}
                                \label{eq 10.211b}
[v_j]_{\beta/\sigma,\beta;Q_R}\le CR^{\sigma+\alpha-\beta},
\end{equation}
where $C$ depends only on $\alpha$ and $\alpha'$. Thus, we can find $v\in C^{1+\alpha/\sigma,\sigma+\alpha}(\bR^{d+1}_0)$ such that $v$ satisfies \eqref{eq 10.211b} for any $R>0$ and $\beta\in [0,\sigma+\alpha]$, and along a subsequence $v_j\to v$ in $C^{\beta/\sigma,\beta}$ locally uniformly for any $\beta\in [0,\sigma+\alpha)$.
\end{lemma}

We remark that \eqref{eq 10.211} will be used below to prove that $v$ satisfies Condition $(iii)$ in Theorem \ref{thm 10.251}, and \eqref{eq 10.211b} will be used to show that $v$ satisfies Condition $(i)$.

\begin{proof}[Proof of Lemma \ref{lem4.3}]
For any $R>0$ and $\beta\in [0,\sigma+\alpha']$,
\begin{align*}
&[v_j]_{\beta/\sigma,\beta;Q_R\cap\{t>-R_j^\sigma\}}\\
&=\frac{[u_j(t_j+r_j^\sigma \cdot,x_j+r_j\cdot)-T_j(t_j+r_j^\sigma \cdot,x_j+r_j\cdot)]_{\beta/\sigma,\beta;Q_R\cap\{t>-R_j^\sigma\}}}{r_j^{\sigma+\alpha}
[u_j]_{1+\alpha/\sigma,\alpha+\sigma;Q^{k_j}}}\\
& \le \frac{r_j^{\beta}[u_j-T_j]_{\beta/\sigma,\beta;
Q_{Rr_j}(t_j,x_j)\cap\{t>-1\}}}{r_j^{\sigma+\alpha}
 [u_j]_{1+\alpha/\sigma,\alpha+\sigma;Q^{k_j}}}\\
& \le\frac{r_j^\beta (Rr_j)^{\sigma+\alpha'-\beta}[u_j]_{1+\alpha'/\sigma,\alpha'+\sigma;
Q_{Rr_j}(t_j,x_j)\cap\{t>-1\}}}{r_j^{\sigma+\alpha}
 [u_j]_{1+\alpha/\sigma,\alpha+\sigma;Q^{k_j}}}\le CR^{\sigma+\alpha-\beta},
 \end{align*}
 where we used \eqref{eq10.16} in the last inequality.

For any $R<R_j$, by the choice of $k_j$ we have
\begin{align*}
& [v_j]_{1+\alpha/\sigma,\alpha+\sigma;Q_R}=\frac{[u_j(t_j+r_j^\sigma \cdot,x_j+r_j\cdot)]_{1+\alpha/\sigma,\alpha+\sigma;Q_R}}{r_j^{\sigma+\alpha}
 [u_j]_{1+\alpha/\sigma,\alpha+\sigma;Q^{k_j}}}\\
&= \frac {[u_j]_{1+\alpha/\sigma,\alpha+\sigma;Q_{Rr_j}(t_j,x_j)}}{
 [u_j]_{1+\alpha/\sigma,\alpha+\sigma;Q^{k_j}}}
\le \frac {[u_j]_{1+\alpha/\sigma,\alpha+\sigma;Q^{k_{j+1}}}}{
 [u_j]_{1+\alpha/\sigma,\alpha+\sigma;Q^{k_j}}}\le 2^{\alpha-\alpha'}.
 \end{align*}
Using  the interpolation inequality, we reach \eqref{eq 10.211b}. The last statement of the lemma follows from the Arzela-Ascoli theorem and the Cauchy diagonal method.
\end{proof}

\begin{lemma}
                            \label{lem4.4}
The function $v$ in Lemma \ref{lem4.3} satisfies the conditions in Theorem \ref{thm 10.251}.
\end{lemma}
\begin{proof}
By Lemma \ref{lem4.3}, Condition $(i)$ is satisfied. Next we verify Condition $(iii)$ for $\sigma\in (1,2)$.
For any  measure $\mu$ with compact support and $\delta\in (0,1)$, we define
\begin{equation*}
V_j(t,x)=\int_{\bR^d}v_j(t,x+h)-v_j(t,x)-\frac{v_j(t,x)-v_j(t,x-\delta h)}{\delta}\,d\mu(h).
\end{equation*}
Since $T_j$ is linear in $t$, from the definition of $v_j$, we have
\begin{align*}
&\partial_tV_j(t,x)
=\frac{r_j^{-\alpha}}{[u_j]_{1+\alpha/\sigma,\alpha+\sigma;Q^{k_j}}}\\
&\quad \cdot
\int_{\bR^d}\Big[\partial_t u_j(t_j+r_j^\sigma t,x_j+r_j (x+h))-\partial_t u_j(t_j+r_j^\sigma t,x_j+r_j x)\\
&\qquad-\frac{\partial_t u_j(t_j+r_j^\sigma t,x_j+r_j x)-\partial_t u_j(t_j+r_j^\sigma t,x_j+r_j (x-\delta h))}{\delta}\Big]\,d\mu(h),
\end{align*}
which is equal to
\begin{align}
&\frac{r_j^{-\alpha}}{[u_j]_{1+\alpha/\sigma,\alpha+\sigma;Q^{k_j}}}
\Big[\int_{\bR^d}\partial_tu_j(t_j+r_j^\sigma t,x_j+r_j (x+h))\,d\mu(h)\nonumber\\ \nonumber
&\quad+
\int_{\bR^d}\frac{\partial_t u_j(t_j+r_j^\sigma t,x_j+r_j (x-\delta h))}{\delta}\,d\mu(h)\\
&\quad -(1+1/\delta)\|\mu\|_{L_1}\partial_t u_j(t_j+r_j^\sigma t,x_j+r_j x)\Big]
.\label{eq 10.233}
\end{align}
For any $a\in \cA_j$, define $\hat K_{a}(y)=r_j^{d+\sigma}K_a(r_jy)$,
which satisfies
\begin{equation*}
\frac{\lambda(2-\sigma)}{|y|^{d+\sigma}}\le \hat K_{a}(y)\le \frac{\Lambda(2-\sigma)}{|y|^{d+\sigma}},
\end{equation*}
and $\hat L_{a}$ be the corresponding operator with kernel $\hat K_{a}$.

Clearly,
\begin{align*}
&\hat L_{a}V_j
=\int_{\bR^d}\Big[\hat L_{a}(v_j(t,x+h)-v_j(t,x))-\frac{\hat L_{a}(v_j(t,x)-v_j(t,x-\delta h))}{\delta}\Big]d\mu(h)\\
&=\frac{r_j^{-\alpha}}{[u_j]_{1+\alpha/\sigma,\alpha+\sigma;Q^{k_j}}}\int_{\bR^d}\Big\{(L_{{a}}u_j)(t_j+r_j^\sigma t,x_j+r_j x+r_j h)\\ \nonumber
&\quad-(L_{{a}}u_j)(t_j+r_j^\sigma t,x_j+r_j x)\\
&\quad-\frac{(L_{{a}}u_j)(t_j+r_j^\sigma t,x_j+r_j x)-(L_{{a}}u_j)(t_j+r_j^\sigma t,x_j+r_j (x-\delta h))}{\delta}\Big\}\,d\mu(h),
\end{align*}
where in the second equality, we used the definitions of $\hat L_a$, $v_j$ and the fact that $T_j$ is at most second-order in $x$ variable, so that for $\sigma>1$ and any $y\in \bR^d$
\begin{equation*}
\delta T_j(t,x+h,y)-\delta T_j(t,x,y)=0.
\end{equation*}
Therefore, for any $(t,x)\in (-R_j^{\sigma},0)\times\bR^d$,
\begin{align*}
&\sup_{a\in \cA_j} \hat L_{a}(V_j)=\frac{r_j^{-\alpha}}{[u_j]_{1+\alpha/\sigma,\alpha+\sigma;Q^{k_j}}}\sup_{a\in \cA_j}\Big\{\int_{\bR^d}\Big((L_{{a}}u_j)(t_j+r_j^\sigma t,x_j+r_j x+r_j h)\\
&\, +\frac{(L_{a}u_j)(t_j+r_j^\sigma t,x_j+r_j (x-\delta h))}{\delta}\Big)d\mu(h)-(1+\frac{1}{\delta})(L_{a}u_j)(t_j+r_j^\sigma t,x_j+r_j x)\Big\}
\end{align*}
\begin{align*}
&=\frac{r_j^{-\alpha}}{[u_j]_{1+\alpha/\sigma,\alpha+\sigma;Q^{k_j}}}\sup_{a\in \cA_j}\Big\{\int_{\bR^d}(L_{a}u_j)(t_j+r_j^\sigma t,x_j+r_j x+r_j h)\\
&\quad +f_a(t_j+r_j^\sigma r,x_j+r_j x+r_j h)+\frac{1}{\delta}\big((L_{a}u_j)(t_j+r_j^\sigma t,x_j+r_j (x-\delta h))\\
&\quad+f_{{a}}(t_j+r_j^\sigma t,x_j+r_j (x-\delta h))\big)\,d\mu(h)
\\
&\quad-(1+\frac{1}{\delta})((L_{a}u_j)(t_j+r_j^\sigma t,x_j+r_j x)+f_a(t_j+r_j^\sigma t,x_j+r_j x))\cdot\|\mu\|_{L_1}\\
&\quad-\int_{\bR^d}\Big(f_a(t_j+r_j^\sigma t,x_j+r_j x+r_j h)-f_a(t_j+r_j^\sigma t,x_j+r_j x)\\
&\quad+\frac{1}{\delta}(f_{a}(t_j+r_j^\sigma t,x_j+r_j (x-\delta h))-f_a(t_j+r_j^\sigma t,x_j+r_j x)\Big)\,d\mu(h)\Big\}.
\end{align*}
Note that for sufficiently large $j$ such that   $\max((-t)^{1/\sigma},|x|+|h|)\le R_j$ whenever $h\in \text{supp} \mu$, we have
\begin{align*}
&|f_a(t_j+r_j^\sigma t,x_j+r_j x+r_j h)-f_a(t_j+r_j^\sigma t,x_j+r_j x)|\le [f_{{a}}]_{\alpha/\sigma,\alpha;Q_1}|r_j h|^{\alpha},\\
&|f_{{a}}(t_j+r_j^\sigma t,x_j+r_j (x-\delta h))-f_{{a}}(t_j+r_j^\sigma t,x_j+r_j x)|\le 
[f_{{a}}]_{\alpha/\sigma,\alpha;Q_1}|\delta r_j h|^{\alpha}.
\end{align*}
Therefore, by the  inequality
$$\sup\{f+g-h\}\ge \inf f+\inf g-\inf h,$$
we have that for $(t,x)\in \bR^{d+1}_0$ and $h\in \text{supp}\,\mu$ so that $\max((-t)^{1/\sigma},|x|+|h|)\le R_j$,
\begin{align}\nonumber
&\sup_{a\in \cA_j}\hat L_{a}(V_j)(t,x)\\ \nonumber
&\ge\frac{r_j^{-\alpha}}{[u_j]_{1+\alpha/\sigma,\alpha+\sigma;Q^{k_j}}}
\Big[\inf_{a \in \cA_j}\int_{\bR^d}(L_{{a}}u_j)(t_j+r_j^\sigma t,x_j+r_j x+r_j h)\\ \nonumber
&\quad+f_a(t_j+r_j^\sigma r,x_j+r_j x+r_j h)\,d\mu(h)\\ \nonumber
&\quad+\inf_{a \in \cA_j}\int_{\bR^d}\frac{1}{\delta}\Big((L_{{a}}u_j)(t_j+r_j^\sigma t,x_j+r_j (x-\delta h))\\ \nonumber
&\quad+f_{{a}}(t_j+r_j^\sigma t,x_j+r_j (x-\delta h))\Big)\,d\mu(h)\\ \nonumber
&\quad-\frac{\delta+1}{\delta}\|\mu\|_{L_1}\inf_a[(L_au_j)(t_j+r_j^\sigma t,x_j+r_j x)+f_a(t_j+r_j^\sigma t,x_j+r_j x)]\\ \label{eq 1.181}
&\quad-(1+\delta^{\alpha-1})r_j^\alpha
\sup_{a\in \cA_j}[f_{a}]_{\alpha/\sigma,\alpha;Q_1}\int_{\bR^d}|h|^\alpha\,d\mu(h)\Big].
\end{align}
Since each $u_j$ satisfies
\begin{equation}\label{eq 10.253}
\partial_t u_j=\inf_{a\in \cA_j}(L_a u_j+f_a)\quad \text{in}\,\, Q_1,
\end{equation}
it follows from \eqref{eq 10.233} and \eqref{eq 1.181} that in any bounded subset of $\bR_0^{d+1}$, for sufficiently large $j$,
\begin{align}
                                            \label{eq3.57}
\partial_t V_j-\sup_{a\in \cA_j}\hat L_{a}V_j\le \frac{1+\delta^{\alpha-1}}{[u_j]_{1+\alpha/\sigma,\alpha+\sigma;Q^{k_j}}}
\sup_{a\in \cA_j}[f_a]_{\alpha/\sigma,\alpha}\int_{\bR^d}|h|^\alpha\,d\mu(h).
\end{align}

We denote
$$
V(t,x):=\int_{\bR^d}\Big[v(t,x+h)-v(t,x)-\frac{v(t,x)-v(t,x-\delta h)}{\delta}\Big]\,d\mu(h).
$$
For fixed $(t,x)\in \bR^{d+1}_0$, by \eqref{eq 10.211} in Lemma \ref{lem4.3} and using the fact that $\mu$ has compact support, we have
\begin{equation}
                            \label{eq4.14}
\lim_j\partial_t V_j(t,x)=\partial_t\lim_j V_j(t,x)=\partial_t V(t,x),
\end{equation}
for $|y|\le 1$,
\begin{equation}
                            \label{eq4.17}
|\delta V_j(t,x,y)|\le C|y|^{\sigma+\alpha'},\quad
|\delta V(t,x,y)|\le C|y|^{\sigma+\alpha'},
\end{equation}
and for $|y|>1$,
\begin{equation}
                                    \label{eq4.18}
|V_j(t,y)|,|V(t,y)|\le C|y|^{\alpha},
\end{equation}
where $C$ depends on $\mu$.
Clearly,
\begin{equation*}
\sup_{a\in \cA_j}\big|L_a(V_j-V)(t,x)\big|
\le C\int_{\bR^d} \big|\delta (V_j-V)(t,x,y)\big||y|^{-d-\sigma}\,dy.
\end{equation*}
It follows from Lemma \ref{lem4.3} that $\delta (V_j-V)\to 0$ locally uniformly. Therefore, by \eqref{eq4.17}, \eqref{eq4.18}, and the dominated convergence theorem, we have
\begin{equation*}
\lim_j\sup_{a\in \cA_j}\big|(L_a(V_j-V)(t,x))\big|=0,
\end{equation*}
i.e.,
\begin{equation}
                            \label{eq4.27}
\lim_j\sup_{a\in \cA_k}L_aV_j(t,x)=\sup_{a\in \cA_k}L_aV(t,x).
\end{equation}
Since $\mu$ has compact support, by Lemma \ref{lem4.2} and \eqref{eq2.26}, we have  $R_j\to \infty$ and
$$
[u_j]_{1+\alpha/\sigma,\alpha+\sigma;Q^{k_j}}\ge 2^{-k_j(\alpha-\alpha')}[u_j]_{1+\alpha/\sigma,\alpha+\sigma;Q^{k_j}}\to \infty.
$$
For fixed $\delta\in (0,1)$, we send $j$ to infinity to get from \eqref{eq3.57}, \eqref{eq4.14}, and \eqref{eq4.27} that
$$
\partial_tV-\cM^+V\le 0\quad \text{in}\,\, \bR^{d+1}_0.
$$
By sending $\delta$ to $0$ and using the dominated convergence theorem, we conclude that
\begin{equation*}
\int_{\bR^d}\big(v(t,x+h)-v(t,x)-h^TDv(t,x)\big)\,d\mu(h)
\end{equation*}
is a subsolution as well. Therefore, for $\sigma>1$, $v$ satisfies Condition $(iii)$.

It remains to verify that $v$ satisfies Condition $(ii)$.
Clearly,  for fixed $(t,x),(s,h)\in \bR_0^{d+1}$, when $j$ is sufficiently large,
\begin{align}\nonumber
&\partial_t\big(v_j(t+s,x+h)-v_j(t,x)\big)
=r_j^{-\alpha}[u_j]_{1+\alpha/\sigma,\alpha+\sigma;Q^{k_j}}^{-1}\\
&\quad\cdot \big(\partial_tu_j(t_j+r_j^\sigma (t+s),x_j+r_j (x+h))-\partial_t u_j(t_j+r_j^\sigma t,x_j+r_j x)\big). \label{eq 10.251}
\end{align}
On the other hand,
\begin{align}\nonumber
&\cM^-\big(v_j(t+s,x+h)-v_j(t,x)\big)
=r_j^{-\alpha}[u_j]_{1+\alpha/\sigma,\alpha+\sigma;Q^{k_j}}^{-1}\\
&\quad \cdot \cM^-\big(u_j(t_j+r_j^\sigma (t+s),x_j+r_j (x+h))-u_j(t_j+r_j^\sigma t,x_j+r_j x)\big).\label{eq 10.252}
\end{align}
Combining \eqref{eq 10.253}, \eqref{eq 10.251}, and \eqref{eq 10.252}, we obtain that for $j$ sufficiently large,
\begin{align*}
&\partial_t\big(v_j(t+s,x+h)-v_j(t,x)\big)
-\cM^-\big(v_j(t+s,x+h)-v_j(t,x)\big)\\
&\ge -[u_j]_{1+\alpha/\sigma,\alpha+\sigma;Q^{k_j}}^{-1}
\sup_a[f_a]_{\alpha/\sigma,\alpha;Q_1}.
\end{align*}
By sending $j$ to infinity, we get for any $(t,x)\in \bR^{d+1}_0$,
\begin{equation*}
\partial_t\big(v(t+s,x+h)-v(t,x)\big)
-\cM^-\big(v(t+s,x+h)-v(t,x)\big)\ge 0.
\end{equation*}
Similarly,
\begin{equation*}
\partial_t\big(v(t+s,x+h)-v(t,x)\big)
-\cM^+\big(v(t+s,x+h)-v(t,x)\big)\le 0.
\end{equation*}
The lemma is proved.
\end{proof}

Now we are ready to finish
\begin{proof}[Proof of Theorem \ref{thm 1104.1}]
By Lemma \ref{lem4.4} and Theorem \ref{thm 10.251}, $v$ is a polynomial of order $\nu$ in $x$ and 1 in $t$. Since at the origin $v_j$ along with its first derivative in $t$ and up to $\nu$-th order derivatives in $x$ are $0$, by Lemma \ref{lem4.3} the same is true for $v$. Therefore, $v\equiv 0$.
This gives us a contradiction with \eqref{eq 1.191} and Lemma \ref{lem4.3}.
The proof is completed.
 \end{proof}

\subsection{Equations with $(t,x)$-dependent kernels}
In this subsection, we consider the case that kernels also depend on $(t,x)$ and H\"older continuous in $(t,x)$, i.e., there exists $A>0$ such that for any $a\in \mathcal{A}$, \eqref{eq 11.051} is satisfied.
We only prove Theorem \ref{thm 1} in the case when $\sigma+\alpha>2$ and the proof of the cases $\sigma+\alpha<2$ is similar and actually simpler. Below we divide the proof into several steps.

Let $\eta$ be a nonnegative smooth cutoff function with $\eta\equiv 1$ in $Q_{1}$ and vanishes outside $(-(5/4)^\sigma,(5/4)^\sigma)\times B_{5/4}$. Set $v:=\eta u$
and note that  in $Q_1$,
\begin{align*}
v_t&=\eta u_t+\eta_t u
=\inf_{a\in \mathcal{A}}(\eta L_au+\eta f_a+\eta_t u)\\
&=\inf_{a\in \mathcal{A}}(L_a^0v+(L_a-L_a^0)v+\eta L_au-L_av + \eta f_a+\eta_tu)
\end{align*}
where
\begin{align*}
L_a^0v = \int_{\bR^d}\delta v(t,x,y)K_a(0,0,y)\,dy.
\end{align*}
We further define
\begin{align*}
h_a&:=\eta L_au-L_a v\\
&=\int_{\bR^d}
\big((\eta(t,x)-\eta(t,x+y))u(t,x+y)+y^TD\eta(t,x)u(t,x)\big)K_a(t,x,y)\,dy
\end{align*}
and
$$
g_a:=(L_a-L_a^0)v=\int_{\bR^d}\delta v(t,x,y)\big(K_a(t,x,y)-K_a(0,0,y)\big)\,dy.
$$
Here in order to apply the argument of freezing the coefficients, we subtracted and added $K_a(0,0,y)$ in the formula above.

\begin{lemma}
                            \label{lem4.5}
Assume that $u\in C^{1+\alpha/\sigma,\sigma+\alpha}(Q_{11/8})\cap C^{\alpha/\sigma,\alpha}((-(11/8)^\sigma,0)\times\bR^d)$. Let $h_{a}$ and $g_{a}$ be functions defined above. Then for any $\alpha\in \cA$, we have
\begin{align}
                                    \label{eq4.19}
[g_{a}]_{\alpha/\sigma,\alpha;Q_1}&\le CA\big([v]_{1+\alpha/\sigma,\alpha+\sigma}+
[v]_{\alpha/\sigma,\alpha}\big),\\
                                    \label{eq4.21}
[h_a]_{\alpha/\sigma,\alpha;Q_1}&\le C(A+1)
\big(\|u\|_{\alpha/\sigma,\alpha;(-(11/8)^\sigma,0)\times\bR^d}+\|D^2 u\|_{L_\infty(Q_{11/8})}\big).
\end{align}
Moreover,
\begin{align}
&[u]_{1+\alpha/\sigma,\alpha+\sigma;Q_{1/2}}\le C\Big([u]_{1+\alpha'/\sigma,\alpha'+\sigma;Q_{5/4}}+C_0
+A[u]_{1+\alpha/\sigma,\alpha+\sigma;Q_{5/4}}\nonumber\\
                \label{eq 1.211}
&+(A+1)(\|u\|_{\alpha/\sigma,\alpha;(-(11/8)^\sigma,0)\times \bR^d}+\|D^2 u\|_{L_\infty(Q_{11/8})})\big).
\end{align}
Here the constant $C$ depends only on $d$, $\lambda$, $\Lambda$, $\alpha$, and $\sigma$, and is uniformly bounded as $\sigma\to 2$.
\end{lemma}
\begin{proof}
For  $(t,x),(t',x')\in Q_1$, set $l=\max(|x-x^\prime|,|t-t^\prime|^{1/\sigma})$. Without loss of generality, we may assume that $l\le 1/4$.

\underline{Estimates of $g_{a}$:}
From the definition and the triangle inequality,
\begin{align*}
&\big|g_{a}(t,x)-g_{a}(t^\prime,x^\prime)\big|\\
&=\Big|\int_{\bR^d}\delta v(t,x,y)\big(K_a(t,x,y)-K_a(0,0,y)\big)\,dy\\
&\quad -\int_{\bR^d}\delta v(t^\prime,x^\prime,y)\big(K_a(t',x^\prime,y)-K_a(0,0,y)\big)\,dy\Big|\\
&\le\Big|\int_{\bR^d}\big(\delta v(t,x,y)-\delta v(t^\prime,x^\prime,y)\big)\big(K_a(t,x,y)-K_a(0,0,y)\big)\,dy\Big|\\
&\quad +\Big|\int_{\bR^d}\delta v(t^\prime,x^\prime,y)\big(K_a(t,x,y)-K_a(t',x^\prime,y)\big)\,dy\Big|\\
&=:\RN{1}+\RN{2}.
\end{align*}
Then we estimate  $\RN{1}$ and $\RN{2}$ separately. First,  similar to \eqref{eq 1104.2}, $\RN{1}$ is less than
\begin{align*}
&\int_{B_l}\Big|\big(\delta v(t,x,y)-\delta v(t^\prime,x^\prime,y)\big)\big(K_a(t,x,y)-K_a(0,0,y)\big)\Big|\,dy\\
&\quad+\int_{\bR^d\setminus B_l}\Big|\big(\delta v(t,x,y)-\delta v(t^\prime,x^\prime,y)\big)\big(K_a(t,x,y)-K_a(0,0,y)\big)\Big|\,dy
:=\RN{1}_1+\RN{1}_2.
\end{align*}
Applying \eqref{eq 8.86}, we have
\begin{align*}
\RN{1}_1&\le C[v]_{1+\alpha/\sigma,\alpha+\sigma;Q_{5/4}} \int_{B_l}l^{\alpha+\sigma-2}|y|^2\big(K_a(t,x,y)-K_a(0,0,y)\big)\,dy\\
&\le  CA(2-\sigma)[v]_{1+\alpha/\sigma,\alpha+\sigma}l^{\alpha+\sigma-2}
\big(|x|^\alpha+|t|^{\alpha/\sigma}\big)\int_{B_l}|y|^2|y|^{-d-\sigma}\,dy\\
&= CAl^\alpha [v]_{1+\alpha/\sigma,\alpha+\sigma}.
\end{align*}
For $\RN{1}_2$,  we have
\begin{align*}
\RN{1}_2&\le C[v]_{1+\alpha/\sigma,\alpha+\sigma}l \int_{\bR^d\setminus B_l} |y|^{\sigma+\alpha-1}\big|K_a(t,x,y)-K_a(0,0,y)\big|\,dy\\
&\le CA(2-\sigma)[v]_{1+\alpha/\sigma,\alpha+\sigma}l^\alpha
\big(|x|^\alpha+|t|^{\alpha/\sigma}\big)
\le CAl^\alpha[v]_{1+\alpha/\sigma,\alpha+\sigma}.
\end{align*}
Next, we bound
\begin{align*}
\RN{2}&\le \int_{\bR^d\setminus B_{1}} \big([v]_{\alpha/\sigma,\alpha}|y|^{\alpha}+\|Dv\|_{L_\infty}|y|\big) \big|K_a(t,x,y)-K_a(t',x^\prime,y)\big|\,dy\\
&\quad +\int_{B_{1}}\|D^2 v\|_{L_\infty}|y|^2
\big|K_a(t,x,y)-K_a(t',x^\prime,y)\big|\,dy\\
&\le CAl^\alpha\big([v]_{\alpha/\sigma,\alpha}+\|Dv\|_{\infty}+\|D^2 v\|_{L_\infty}\big).
\end{align*}
Combining the estimates of $\RN{1}$, $\RN{2}$,  and the interpolation inequality, we get
\begin{equation*}
|g_{a}(t,x)-g_{a}(t^\prime,x^\prime)|\le C Al^\alpha  \big ([v]_{1+\alpha/\sigma,\alpha+\sigma}+[v]_{\alpha/\sigma,\alpha} \big),
\end{equation*}
which implies \eqref{eq4.19}.

\underline{Estimates of $h_a$:} For simplicity of notation, we denote
\begin{equation}
                            \label{eq11.32}
\xi(t,x,y)=\big(\eta(t,x)-\eta(t,x+y)\big)u(t,x+y)+y^TD\eta(t,x)u(t,x).
\end{equation}
By the Leibniz rule, we have
\begin{align}\nonumber
\xi(t,x,y)&=y^T\int_{0}^1\big(D\eta(t,x)u(t,x)-D\eta(t,x+sy)u(t,x+y)\big)\,ds\\
&=-\int_0^1\int_{0}^1\big(u(t,x)s y^TD^2 \eta(t,x+s'sy)y\nonumber\\
\label{eq 11.222}
&\quad +y^TD\eta (t,x+sy)y^TDu(t,x+s'y)
\big)\,ds'\,ds,
\end{align}
which implies that when $|y|\le 1/8$,
\begin{equation} \label{eq 11.2210}
|\xi(t,x,y)|\le C|y|^2\big(\|u\|_{L_\infty(Q_{11/8})}+\|D u\|_{L_\infty(Q_{11/8})}\big).
\end{equation}
On the other hand, clearly when $|y|\ge 1/8$,
\begin{equation}
|\xi(t,x,y)|
\le C\big(\|u\|_{L_\infty((-1,0)\times \bR^d)}+|y|\|u\|_{L_\infty(Q_1)}\big).\label{eq 12.261}
\end{equation}

Note that
\begin{align}\nonumber
&\big|h_a(t,x)-h_a(t',x')\big|\le \int_{\bR^d}\big|\xi(t,x,y)-\xi(t',x',y)\big|K_a(t,x,y)\,dy\\
&\quad +\int_{\bR^d}|\xi(t',x',y)|\big|K_a(t,x,y)-K_a(t',x',y)\big|\,dy
=:\RN{3}+\RN{4}. \label{eq 1125.1}
\end{align}
\underline{Estimate of $\RN{3}$:} By \eqref{eq 11.222} when $|y|\le 1/8$, we have
\begin{align*}
&\big|\xi(t,x,y)-\xi(t',x',y)\big|\\
&=|\int_0^1\int_0^1s\big(u(t',x')y^TD^2\eta(t',x'+ss'y)y-u(t,x)y^TD^2 \eta(t,x+ss'y)y\big)\,dx\,ds'\\
&\quad+\int_{0}^1\int_0^1\Big[y^TD\eta(t',x'+sy)y^TDu(t',x'+s'y)\\
&\qquad\qquad\qquad-y^TD\eta(t,x+sy)y^TDu(t,x+s'y)\Big]\,ds\,ds'|\\
&\le C|y|^2l^\alpha\big([u]_{\alpha/\sigma,\alpha;Q_{1}}+\|u\|_{L_\infty(Q_{1})}\big)
+C|y|^2l\big(\|D^2u\|_{L_\infty(Q_{11/8})}+\|D u\|_{L_\infty(Q_{11/8})}\big)\\
&\le C|y|^2l^\alpha\big(\|u\|_{L_\infty(Q_{11/8})}+\|D^2u\|_{L_\infty(Q_{11/8})}\big),
\end{align*}
where we used the interpolation inequalities in the last inequality.
On the other hand, when $|y|>1/8$,
\begin{equation*}
\big|y^TD\eta(t,x)u(t,x)-y^TD\eta(t',x')u(t',x')\big|\le |y|l^\alpha\|u\|_{\alpha/\sigma,\alpha;Q_{1}}
\end{equation*}
and
\begin{align*}
&\big|(\eta(t,x)-\eta(t,x+y))u(t,x+y)-(\eta(t',x')-\eta(t',x'+y))u(t',x'+y)\big|\\
&=\big| u(t,x+y)\big(\eta(t,x)-\eta(t',x')-\eta(t,x+y)+\eta(t',x'+y)\big)\\
&\quad +\big(\eta(t',x')-\eta(t',x'+y)\big)\big(u(t,x+y)-u(t',x'+y)\big)\big|\\
&\le C(l\|u\|_{L_\infty((-1,0)\times \bR^d)}+l^\alpha\|u\|_{\alpha/\sigma,\alpha;(-1,0)\times \bR^d}),
\end{align*}
which imply that when $|y|>1/8$,
\begin{align*}
&\big|\xi(t,x,y)-\xi(t',x',y)\big|\\
&\le C\big(l\|u\|_{L_\infty((-1,0)\times \bR^d)}+l^\alpha\|u\|_{\alpha/\sigma,\alpha;(-1,0)\times \bR^d}\big)
+C|y|l^\alpha\|u\|_{\alpha/\sigma,\alpha;Q_1}.
\end{align*}
Now with the above estimates,  we obtain
\begin{align*}
\RN{3}&\le \int_{B_{1/8}}C|y|^2l^\alpha \big(\|u\|_{L_\infty(Q_{11/8})}+\|D^2 u\|_{L_\infty(Q_{11/8})}\big)K_a(t,x,y)\,dy\\
&\quad +Cl^\alpha\|u\|_{\alpha/\sigma,\alpha;Q_1}
\int_{B_{1/8}^c}|y|K_a(t,x,y)\,dy\\
&\quad +C\big(l\|u\|_{L_\infty((-1,0)\times \bR^d)}+l^\alpha\|u\|_{\alpha/\sigma,\alpha;(-1,0)\times \bR^d}\big)
\int_{B_{1/8}^c}K_a(t,x,y)\,dy\\
&\le Cl^\alpha\big(\|D^2 u\|_{L_\infty(Q_{11/8})}+\|u\|_{\alpha/\sigma,\alpha;(-(11/8)^\sigma,0)\times \bR^d}\big).
\end{align*}

\underline{Estimate of $\RN{4}$:} By  \eqref{eq 11.2210} and \eqref{eq 12.261}, we have
\begin{align*}
\RN{4}\le Cl^\alpha A\big(\|u\|_{L_\infty((-(11/8)^\sigma,0)\times \bR^d)}+\|D u\|_{L_\infty(Q_{11/8})}\big).
\end{align*}

The estimates of $\RN{3}$ and $\RN{4}$ with the interpolation inequalities give \eqref{eq4.21}.

Now we apply Theorem \ref{thm 1104.1} to $v$ with the estimates of $g_a$ and $h_a$ in Lemma \ref{lem4.5} to obtain
\begin{align*}\nonumber
&[v]_{1+\alpha/\sigma,\alpha+\sigma;Q_{1/2}}\le C\Big([v]_{1+\alpha'/\sigma,\alpha'+\sigma}
+A[v]_{1+\alpha/\sigma,\alpha+\sigma}+A[v]_{\alpha/\sigma,\alpha}\\
&\quad +(A+1)(\|u\|_{\alpha/\sigma,\alpha;(-(11/8)^\sigma,0)\times \bR^d}+\|D^2 u\|_{L_\infty(Q_{11/8})})+\sup_a[\eta f_a]_{\alpha/\sigma,\alpha;Q_1}\Big).
\end{align*}
Since $\eta\equiv 1$ in $Q_1$ and has compact support in $(-(5/4)^\sigma,(5/4)^\sigma)\times B_{5/4}$, we get \eqref{eq4.21}.
The lemma is proved.
\end{proof}

\begin{proof}[Proof of Theorem \ref{thm 1}]

We first use a scaling argument. For any $\epsilon>0$, set  $\hat{u}(t,x):=\epsilon^{-\sigma}u(\epsilon^\sigma t,\epsilon x)$. Since $u$ satisfies \eqref{eq 1110.1},  we have
\begin{equation*}
\hat{u}_t(t,x)=\inf_a\Big\{\int_{\bR^d}\delta \hat{u}(t,x,y)K^\epsilon_a(t,x,y)\,dy+f_a(\epsilon^\sigma t,\epsilon x)\Big\}\quad \text{in}\,\, Q_{1/\epsilon},
\end{equation*}
where
$$K^\epsilon_a(t,x,y)=\epsilon^{d+\sigma}K_a(\epsilon^\sigma t,\epsilon x,\epsilon y).$$
Clearly,
\begin{equation*}
\big|K_a^\epsilon(t,x,y)-K_a^\epsilon(t',x',y)\big|\le A(2-\sigma)\epsilon^\alpha\big(|x-x'|^\alpha+|t-t'|^{\alpha/\sigma}\big)\frac{\Lambda}{|y|^{d+\sigma}}.
\end{equation*}
Then we apply \eqref{eq 1.211} to $\hat{u}$ and get
\begin{align*}
[\hat{u}]_{1+\alpha/\sigma,\alpha+\sigma;Q_{1/2}}&\le C\Big([\hat{u}]_{1+\alpha'/\sigma,\alpha+\sigma;Q_{5/4}}+ C_0\epsilon^\alpha+ A\epsilon^\alpha[\hat{u}]_{1+\alpha/\sigma,\alpha+\sigma;Q_{5/4}}\\
&
+( A\epsilon^\alpha+1)\big(\|\hat{u}\|_{\alpha/\sigma,\alpha;(-(11/8)^\sigma,0)\times \bR^d}+\|D^2 \hat{u}\|_{L_\infty(Q_{11/8})}\big)\big).
\end{align*}
Returning back to $u$, we have
\begin{align*}
&[u]_{1+\alpha/\sigma,\sigma+\alpha;Q_{\epsilon/2}}\le C\Big(\epsilon^{\alpha'-\alpha}[u]_{1+\alpha'/\sigma,\alpha'+\sigma;Q_{5\epsilon/4}}
+C_0+A\epsilon^\alpha[u]_{1+\alpha/\sigma,\alpha+\sigma;Q_{5\epsilon/4}}\\
&\quad
+(A\epsilon^\alpha +1)\big(\epsilon^{-\sigma-\alpha}\|u\|_{\alpha/\sigma,\alpha;(-(11\epsilon/8)^\sigma,0)\times \bR^d}
+\epsilon^{2-\sigma-\alpha}\|D^2 u\|_{L_\infty(Q_{11\epsilon/8})}\big)\Big).
\end{align*}
By a translation of the coordinates, the inequality above holds for any $(t,x)\in Q_1$ for sufficiently small $\epsilon>0$
\begin{align}\nonumber
&[u]_{1+\alpha/\sigma,\sigma+\alpha;Q_{\epsilon/2}(t,x)}\\
                                \nonumber
&\le C\Big(\epsilon^{\alpha'-\alpha}[u]_{1+\alpha'/\sigma,\alpha'+\sigma;Q_{5\epsilon/4}(t,x)}+
C_0+A\epsilon^\alpha
[u]_{1+\alpha/\sigma,\alpha+\sigma;Q_{5\epsilon/4}(t,x)}\\ \label{eq 1.252}
&\quad +
(A\epsilon^\alpha+1)\big(\epsilon^{-\sigma-\alpha}
\|u\|_{\alpha/\sigma,\alpha;(t-(11\epsilon/8)^\sigma,t)\times \bR^d}+\epsilon^{2-\sigma-\alpha}\|D^2 u\|_{L_\infty(Q_{11\epsilon/8}(t,x))}\big)\Big).
\end{align}
Let $Q^k$ be defined in \eqref{eq 3.11}. It is obvious that $Q^k$ monotonically increases to $Q_1$. Then for any $(t,x),(s,y)\in Q^k$ such that $t\ge s$, we set $l:=\max(|t-s|^{1/\sigma},|x-y|)$. When $l\ge \epsilon/2$,
\begin{align*}
&\frac{|D^2u(t,x)-D^2u(s,y)|}{l^{\sigma+\alpha-2}}+\frac{|u_t(t,x)-u_t(s,y)|}{l^{\sigma+\alpha-2}}\\
&\le 2^{\sigma+\alpha-1}\epsilon^{2-\sigma-\alpha}
\big(\|u_t\|_{L_\infty(Q^k)}+\|D^2 u\|_{L_\infty(Q^k)}\big);
\end{align*}
when $l<\epsilon/2$,
\begin{align*}
\frac{|D^2u(t,x)-D^2u(s,y)|}{l^{\sigma+\alpha-2}}+\frac{|u_t(t,x)-u_t(s,y)|}{l^{\sigma+\alpha-2}}\le 2[u]_{1+\alpha/\sigma,\alpha+\sigma;Q_{\epsilon/2}(t,x)}.
\end{align*}
Now we choose $\epsilon=2^{-k-2}$ so that for any $(t,x)\in Q^k$, $Q_{11\epsilon/8}(t,x)\subset Q^{k+1}$ and $(t-(11\epsilon/8)^\sigma,t)\subset (-1,0)$.
Combining the two inequalities above with \eqref{eq 1.252}, we obtain
\begin{align}\nonumber
&[u]_{1+\alpha/\sigma,\alpha+\sigma;Q^k}\le 2^{(k+3)(\sigma+\alpha-2)+1}(\|u_t\|_{L_\infty(Q^k)}+\|D^2 u\|_{L_\infty(Q^k)})\\ \nonumber
&\quad+C\Big(2^{(k+2)(\alpha-\alpha')}[u]_{1+\alpha'/\sigma,\alpha'+\sigma;Q^{k+1}}
+C_0+2^{-(k+2)\alpha}A
[u]_{1+\alpha/\sigma,\alpha+\sigma;Q^{k+1}}\\ \label{eq 1.253}
&\quad+(2^{-(k+2)\alpha}A+1)\big(2^{(k+2)(\sigma+\alpha)}
\|u\|_{\alpha/\sigma,\alpha;(-1,0)\times \bR^d}\nonumber\\
&\quad+2^{(k+2)(\sigma+\alpha-2)}\|D^2 u\|_{L_\infty(Q^{k+1})}\big)\Big).
\end{align}
By the interpolation inequalities
\begin{align*}
&[u]_{1+\alpha^\prime/\sigma,\alpha^\prime+\sigma;Q^{k+1}}\le 2^{-2(k+2)(\alpha-\alpha')}[u]_{1+\alpha/\sigma,\alpha+\sigma;Q^{k+1}}+C2^{2(k+2)(\sigma+\alpha')}
\|u\|_{L_\infty},\\
&\|D^2u\|_{L_\infty(Q^{k+1})}+\|u_t\|_{L_\infty(Q^{k+1})}\\
&\le 2^{-2(k+1)(\sigma+\alpha-2)}[u]_{1+\alpha/\sigma,\alpha+\sigma;Q^{k+1}}+C2^{4(k+1)}
\|u\|_{L_\infty(Q^{k+1})},
\end{align*}
we reorganize the  right-hand side of \eqref{eq 1.253} to get
\begin{align*}
&[u]_{1+\alpha/\sigma,\alpha+\sigma;Q^k}\le C\Big(\big(2^{-(k+1)(\sigma+\alpha-2)}+2^{-(k+2)(\alpha-\alpha')}\big)
[u]_{1+\alpha/\sigma,\alpha+\sigma;Q^{k+1}}\\
&\quad+2^{5k(\sigma+\alpha)}\|u\|_{\alpha/\sigma,\alpha;(-1,0)\times \bR^d}
+C_0\Big),
\end{align*}
where $C$ depends on $A$. Obviously, $\sigma+\alpha<3$ and there exists a constant $k_0$ depends on $d$, $\sigma_0$, $\alpha$, $\lambda$, $\Lambda$, and $A$ such that $Q_{1/2}\subset Q^{k_0}$ and for any $k\ge k_0$,
$$
C\big(2^{-(k+1)(\sigma+\alpha-2)}+2^{-2(k+1)(\alpha-\alpha')}\big)<2^{-16}.
$$
Therefore, we have for any $k\ge k_0$,
\begin{align*}
[u]_{1+\alpha/\sigma,\alpha+\sigma;Q^{k}}\le 2^{-16}[u]_{1+\alpha/\sigma,\alpha+\sigma;Q^{k+1}}+C2^{15k}
\|u\|_{\alpha/\sigma,\alpha}+CC_0.
\end{align*}
We multiply both sides above by $2^{-16(k-k_0)}$ and then sum from $k=k_0$ to infinity and obtain that
\begin{align*}
[u]_{1+\alpha/\sigma,\alpha+\sigma;Q^{k_0}}\le C2^{15k_0}\|u\|_{\alpha/\sigma,\alpha;(-1,0)\times \bR^d}+CC_0.
\end{align*}
In particular,
\begin{equation*}
[u]_{1+\alpha/\sigma,\alpha+\sigma;Q_{1/2}}\le C(\|u\|_{\alpha/\sigma,\alpha;(-1,0)\times \bR^d}+C_0).
\end{equation*}
The proof is completed.
\end{proof}

\subsection{An improved estimate} By a more careful analysis, we obtain the following corollary when the kernels depend only on $y$.
\begin{corollary}\label{cor 1125.1}
Let $\sigma\in (0,2)$ and $0<\lambda\le \Lambda$. Assume that for any $a\in\mathcal{A}$, $K_a$ only depends on $y$. There is a constant $\hat{\alpha}\in (0,1)$ depending on $d$, $\sigma$, $\lambda$, and $\Lambda$ (uniformly as $\sigma\to 2$) so that the following holds. Let $\alpha\in (0,\hat{\alpha})$. 
 Suppose $u\in C^{1+\alpha/\sigma,\sigma+\alpha}(Q_1)\cap C^{\alpha/\sigma,\alpha}_{\text{loc}}([-1,0]\times\bR^d)$ is a solution of
\begin{equation*}
u_t=\inf_{a\in \mathcal{A}}(L_a u+f_a)\quad \text{in}\,\, Q_1.
\end{equation*}
Then,
\begin{align}
&[u]_{1+\alpha/\sigma,\alpha+\sigma,Q_{1/2}}\nonumber\\
&\le C[u]_{\alpha/\sigma,\alpha;(-1,0)\times B_2}+C\sum_{j=2}^\infty2^{-j\sigma}[u]_{\alpha/\sigma,\alpha;(-1,0)\times (B_{2^{j}}\setminus B_{2^{j-1}})}+CC_0,
                        \label{eq1.04}
\end{align}
where $C_0=\sup_a[f_a]_{\alpha/\sigma,\alpha; Q_{1}}$ and $C>0$ depends only on $d$, $\lambda$, $\Lambda$, $\alpha$, and $\sigma$, and is uniformly bounded as $\sigma\to 2$.
\end{corollary}
\begin{proof}
Since the proof is quite similar to the proof of Theorem \ref{thm 1}, we only provide a sketch here. By a standard scaling and covering argument, we may assume that
$u\in C^{1+\alpha/\sigma,\sigma+\alpha}(Q_2)\cap C^{\alpha/\sigma,\alpha}_{\text{loc}}([-2,0]\times\bR^d)$ and the equation is satisfied in $Q_2$. Let $\eta$ be a cutoff function such that $\eta\in C_0^
\infty((-2^\sigma,2^\sigma)\times B_2)$ and $\eta\equiv 1$ in $Q_{5/4}$. Let $v=\eta u$, which satisfies
\begin{align*}
v_t=\inf_a(L_av+h_a+\eta f_a+\eta_t u),
\end{align*}
where
$$
h_a=\int_{\bR^d}\xi(t,x,y)K_a(y)\,dy
$$
and $\xi$ is defined in \eqref{eq11.32}.
It is sufficient to estimate $[h_a]_{\alpha/\sigma,\alpha;Q_{1}}$. Since $K_a$ only depends on $y$, it follows that
\begin{align*}
|h_a(t,x)-h_a(t^\prime,x')|=\RN{3},
\end{align*}
where $\RN{3}$ is defined in \eqref{eq 1125.1}.
The estimate is similar to the one in the proof of Lemma \ref{lem4.5}. For any $(t,x),(t',x')\in Q_1$, since $\eta\equiv 1$ in $Q_{5/4}$, $D\eta(t,x)=D\eta(t',x')=0$.
When $|y|\le 1/4$, $\xi(t,x,y)=0$;
When $|y|>1/4$, we have
 \begin{align*}
&\big|\xi(t,x,y)-\xi(t',x',y)\big|\\
&=\big|(\eta(t,x)-\eta(t,x+y))u(t,x+y)+y^TD\eta(t,x)u(t,x)\\
&\quad -(\eta(t',x')-\eta(t',x'+y))u(t',x'+y)+y^TD\eta(t',x')u(t',x')\big|\\
&\le \big|\eta(t,x+y)u(t,x+y)-\eta(t',x'+y)u(t',x'+y)\big|
+\big|u(t,x+y)-u(t',x'+y)\big|.
\end{align*}
Recall that $l = \max\{|x-x^\prime|,|t-t^\prime|^{1/\sigma}\}$. Combining the estimate above, we obtain
\begin{align*}
&\RN{3}=\int_{B_{1/4}^c}\big|\xi(t,x,y)-\xi(t',x',y)\big|K_a(y)\,dy\\
&\le Cl^\alpha\|u\|_{\alpha/\sigma,\alpha;(-1,0)\times B_2}+\int_{B_{1/4}^c}\big|u(t,x+y)-u(t',x'+y)\big|K_a(y)\,dy\\
&\le Cl^\alpha\|u\|_{\alpha/\sigma,\alpha;(-1,0)\times B_2}+\sum_{j=-1}^\infty\int_{B_{2^j}\setminus B_{2^{j-1}}}|u(t,x+y)-u(t',x'+y)|K_a(y)\,dy\\
&\le Cl^\alpha\|u\|_{\alpha/\sigma,\alpha;(-1,0)\times B_2}+Cl^\alpha\sum_{j=-1}^\infty 2^{-j\sigma}[u]_{\alpha/\sigma,\alpha;(-1,0)\times B_{2^{j+1}}}\\
&\le Cl^\alpha\Big(\|u\|_{\alpha/\sigma,\alpha;(-1,0)\times B_2}+\sum_{j=1}^\infty 2^{-j\sigma}[u]_{\alpha/\sigma,\alpha;(-1,0)\times(B_{2^j}\setminus B_{2^{j-1}})}\Big),
\end{align*}
which implies that
\begin{equation*}
[h_a]_{\alpha/\sigma,\alpha;Q_{1}}\le C\Big(\|u\|_{\alpha/\sigma,\alpha;(-1,0)\times B_2}+\sum_{j=1}^\infty 2^{-j\sigma}[u]_{\alpha/\sigma,\alpha;(-1,0)\times(B_{2^j}\setminus B_{2^{j-1}})}\Big).
\end{equation*}
Then we apply Theorem \ref{thm 1} to $v$ and obtain
\begin{align*}
&[v]_{1+\alpha/\sigma,\sigma+\alpha;Q_{1/2}}
\le C\Big(\|v\|_{\alpha/\sigma,\alpha;(-1,0)\times \bR^d}+\|u\|_{\alpha/\sigma,\alpha;(-1,0)\times B_2}\\
&\quad+\sum_{j=1}^\infty 2^{-j\sigma}[u]_{\alpha/\sigma,\alpha;(-1,0)\times(B_{2^j}\setminus B_{2^{j-1}})}+C_0\Big).
\end{align*}
Combining the fact that $\eta\equiv 1$ in $Q_{5/4}$ and replacing $u$ by $u-u(0,0)$, we reach \eqref{eq1.04}.
Therefore, the proof is completed.
\end{proof}

\section{Equations with bounded inhomogeneous terms}

In this section, we present an application of Corollary \ref{cor 1125.1} to nonlocal parabolic equations with merely bounded nonhomogeneous terms:
\begin{equation}
u_t=\inf_{a\in \mathcal{A}}(L_au+f_a),\label{eq 1125.3}
\end{equation}
where $\sup_a\|f_a\|_{L_\infty}<\infty$ and
\begin{equation*}
L_au(x)=\int_{\bR^d}\delta u(t,x,y)K_a(y)\,dy.
\end{equation*}
Before proving Theorem \ref{thm 2}, we first present an interpolation inequality involving the Zygmund semi-norm. The proof can be found in the appendix
\begin{lemma}\label{lemma 12.31}
Let $\alpha\in (0,1)$ and $f\in \Lambda^1((-1,0))\cap L_\infty((-1,0))$. Then we have $f\in C^\alpha((-1,0))$ and
\begin{equation}
                                    \label{eq5.01}
[f]_{\alpha;(-1,0)}\le C\|f\|_{L_\infty((-1,0))}+C[f]_{\Lambda^1((-1,0))},
\end{equation}
where $C$ depends only on $\alpha$.
\end{lemma}

In the sequel, we set
$$
[u]^t_{\Lambda^1}:=[u]^t_{\Lambda^1(\bR^{d+1}_0)},\quad [u]^\ast_{\sigma}:=[u]^\ast_{\sigma;\bR^{d+1}_0}, \quad \text{and}\quad [Du]^t_{\frac{\sigma-1}{\sigma}}=:[Du]^t_{\frac{\sigma-1}{\sigma};\bR^{d+1}_0}.
$$

Let $\eta$ be a smooth even nonnegative function in $\bR$ with unit integral and vanishing outside $(-1,1)$. For $R>0$, we define the mollification of $u$ with respect to $t$ as
$$
u^{(R)}(t,x)=\int_{\bR} \big(2u(t-R^\sigma s,x)-u(t-2R^\sigma s,x)\big)\eta(s-1)\,ds.
$$

The following lemmas will also be used in our proof. We present their proofs in the appendix.

\begin{lemma}
                            \label{lem5.2}
Let $\beta\in (0,1]$ and $R>0$. Then we have
\begin{equation}
                                    \label{eq6.58}
[\partial_t u^{(R)}]^t_{\beta;\bR^{d+1}_0}\le C(\beta)R^{-\beta\sigma}[u]^t_{\Lambda^1}.
\end{equation}
\end{lemma}

\begin{lemma}
                                \label{lem5.5}
Let $\sigma\in (1,2)$, $\alpha\in (0,1)$, and $R>0$ be constants. Assume that  $u$ defined on $\bR^{d+1}_0$ is $C^\sigma$ in $x$, $\Lambda^1$ in $t$, and $Du$ is $C^{(\sigma-1)/\sigma}$ in $t$.
Let $p=p(t,x)$ be the first-order Taylor expansion of $u^{(R)}$ at the origin.
Then for any integer $j\ge 0$, we have
\begin{align}\nonumber
&[u-p]^\ast_{\alpha;(-R^\sigma,0)\times  B_{2^jR}}\\
&\le C(2^jR)^{\sigma-\alpha}[u]^\ast_{\sigma}+C2^{(1-\alpha)j}R^{\sigma-\alpha}
[Du]^t_{\frac{\sigma-1}{\sigma}}+C2^{-j\alpha}R^{\sigma-\alpha}[u]^t_{\Lambda^1} \label{eq 12.62}
\end{align}
and
\begin{align}\nonumber
 [u-p]^t_{\alpha/\sigma;(-R^\sigma,0)\times B_{2^jR}}&\le C2^{j(\sigma-\alpha/2)}
R^{\sigma-\alpha}[u]^\ast_{\sigma}+C2^{j(\sigma+1-\alpha)/2}R^{\sigma-\alpha}
[Du]^t_{\frac{\sigma-1}{\sigma}}\\
&\quad
+C2^{j(\sigma-\alpha/2)}R^{\sigma-\alpha}[u]^t_{\Lambda^1}
 \label{eq 12.61},
 \end{align}
 where $C>0$ is a constant depending only on $d$, $\sigma$, and $\alpha$.
\end{lemma}


In the case when $\sigma=1$, we define $u^{(R)}$ differently. Let $\zeta\in C_0^\infty(B_1)$ be a radial nonnegative function with unit integral. For $R>0$, we define
$$
u^{(R)}(t,x)=\int_{\bR^{d+1}} \big(2u(t-Rs,x-Ry)-u(t-2Rs,x-Ry)\big)\eta(s-1)\zeta(y)\,dy\,ds.
$$
\begin{lemma}
                                \label{lem5.5b}
Let $\alpha\in (0,1)$, and $R>0$ be constants. Assume that  $u$ defined on $\bR^{d+1}_0$ is $\Lambda^1$ in $(t,x)$.
Let $p=p(t,x)$ be the first-order Taylor expansion of $u^{(R)}$ at the origin.
Then for any integer $j\ge 0$, we have
\begin{equation}
[u-p]_{\alpha,\alpha;(-R,0)\times  B_{2^jR}}
\le C2^{j(1-\alpha/2)}R^{1-\alpha}[u]_{\Lambda^1}, \label{eq 12.62b}
\end{equation}
where $C>0$ is a constant depending only on $d$ and $\alpha$.
\end{lemma}

Define $\cP_0$ to be the set of first-order polynomials of $t$, and $\cP_1$ to be the set of first-order polynomials of $(t,x)$.
\begin{lemma}
                            \label{lem5.4}
(i) When $\sigma\in (0,1)$, we have
\begin{equation*}
 [u]^t_{\Lambda^1}+[u]^\ast_{\sigma}\le C\sup_{r>0}\sup_{(t,x)\in \bR_0^{d+1}} r^{-\sigma}\inf_{p\in \mathcal{P}_0}\|u-p\|_{L_\infty(Q_r(t,x))},
 \end{equation*}
where $C>0$ is a constant depending only on $d$ and $\sigma$.

(ii) When $\sigma\in (1,2)$, we have
\begin{equation*}
 [u]^t_{\Lambda^1}+[u]^\ast_{\sigma}+[Du]^t_{\frac{\sigma-1}{\sigma}}\le C\sup_{r>0}\sup_{(t,x)\in \bR_0^{d+1}} r^{-\sigma}\inf_{p\in \mathcal{P}_1}\|u-p\|_{L_\infty(Q_r(t,x))},
 \end{equation*}
where $C>0$ is a constant depending only on $d$ and $\sigma$.

(iii) We have
\begin{equation*}
 [u]_{\Lambda^1}\le C\sup_{r>0}\sup_{(t,x)\in \bR_0^{d+1}} r^{-1}\inf_{p\in \mathcal{P}_1}\|u-p\|_{L_\infty(Q_r(t,x))},
 \end{equation*}
where $C>0$ is a constant depending only on $d$.
\end{lemma}
\begin{proof}
The estimates of $[u]^\ast_{\sigma}$ and $[Du]^t_{\frac{\sigma-1}{\sigma}}$ are standard. See, for instance, \cite[Section 3.3]{Kry97}. We only consider $[u]^t_{\Lambda^1}$. For any polynomial $p$ which is linear in $t$, by the triangle inequality,
 \begin{align*}
 &\big|u(t+s,x)+u(t-s,x)-2u(t,x)\big|\\
 &=\big|u(t+s,x)-p(t+s,x)+u(t-s,x)-p(t-s,x)-2(u(t,x)-p(t,x))\big|\\
&\le 4\|u-p\|_{L_\infty(Q_r(t+s,x))},
\end{align*}
where $r^\sigma=2s$. 
Since $p$ is arbitrary, the inequality above implies that
\begin{equation*}
\big|u(t+s,x)+u(t-s,x)-2u(t,x)\big|\le 8sr^{-\sigma}\inf_p\|u-p\|_{L_\infty(Q_r(t+s,x))}.
\end{equation*}
Similarly, we can prove Assertion (iii). The lemma is proved.
\end{proof}

\begin{proof}[Proof of Theorem \ref{thm 2}]
We only treat the case when $\sigma\ge 1$. For the case when $\sigma<1$, the proof is almost the same with minor modifications.


First we assume $\sigma>1$. Let $\hat\alpha$ be the constant in Corollary \ref{cor 1125.1} and $\alpha\in (0,\hat\alpha)$ be such that $\alpha+\sigma\le 2$.
Let $R>0$ be a constant, $p$ be defined as in Lemma \ref{lem5.5}, and $c_0=\partial_t p$. Let
$K\ge 2\|u-p\|_{L_\infty(Q_{2R})}$ be a constant to be specified later
and denote
$$
g_K=\max\big(\min(u-p,K),-K\big).
$$
Clearly, $g_K\in C^{\alpha/\sigma,\alpha}$ and any $C^{\alpha/\sigma,\alpha}$ norm (or semi-norm) of $g_K$ is less than or equal to that of $u-p$.

Let $v_K$ be the solution to
\begin{equation}
                                \label{eq10.54}
\begin{cases}
 \partial_t v_K=\inf_a(L_a v_K-c_0)\quad &\text{in}\,\, Q_{2R},\\
v_K=g_K \quad&\text{in} \,\, \bR^{d+1}_0\setminus Q_{2R}.
\end{cases}
\end{equation}
Noting that $g_k$ is H\"older continuous in both $t$ and $x$, the solvability follows from Theorem \ref{thm 1} and a regularization argument; see \cite{CS11, Serra15}.
We apply  Corollary  \ref{cor 1125.1} to $v_K$ with a scaling to get
\begin{align}\nonumber
&[v_K]_{1+\alpha/\sigma,\alpha+\sigma;Q_{R/2}}\\
 &\le C\Big(R^{-\sigma}[v_K]_{\alpha/\sigma,\alpha;(-R^\sigma,0)\times B_{2R}}+\sum_{j=2}^\infty 2^{-j\sigma}R^{-\sigma}[v_K]_{\alpha/\sigma,\alpha; (-R^\sigma,0)\times (B_{2^{j}R}\setminus B_{2^{j-1}R})}\Big)\nonumber\\
 &\le C\Big(R^{-\sigma}[v_K]_{\alpha/\sigma,\alpha;(-R^\sigma,0)\times B_{2R}}\nonumber\\
 &\qquad\quad+\sum_{j=2}^\infty 2^{-j\sigma}R^{-\sigma}[u-p]_{\alpha/\sigma,\alpha;(-R^\sigma,0)\times(B_{2^jR}\setminus B_{2^{j-1}R})}\Big),
                            \label{eq 1126.1}
\end{align}
where in the last equality we used the fact that $v_K=g_K$ in $(-(2R)^\sigma,0)\times B_{2R}^c$ and the H\"older norm of $g_K$ is less than the H\"older norm of $u-p$.

By Lemma \ref{lem5.5},
\begin{align}
&[u-p]_{\alpha/\sigma,\alpha;(-R^\sigma,0)\times B_{2^jR}}
\le C2^{j(\sigma-\alpha/2)}R^{\sigma-\alpha}[u]^t_{\Lambda^1}\nonumber\\
&\quad +C2^{j(\sigma-\alpha/2)}R^{\sigma-\alpha}
[u]^\ast_{\sigma}+C2^{j(\sigma+1-\alpha)/2}R^{\sigma-\alpha}
[Du]^t_{\frac{\sigma-1}{\sigma}},       \label{eq1.56}
\end{align}
which together with \eqref{eq 1126.1} gives
\begin{align}
&[v_K]_{1+\alpha/\sigma,\alpha+\sigma;Q_{R/2}}\nonumber\\
 &\le C\big(R^{-\sigma}[v_K]_{\alpha/\sigma,\alpha;(-R^\sigma,0)\times B_{2R}}+R^{-\alpha}[u]^t_{\Lambda^1}+R^{-\alpha}[u]^\ast_{\sigma}
 +R^{-\alpha}[Du]^t_{\frac{\sigma-1}{\sigma}}\big).\label{eq 12.121}
\end{align}

Next we estimate $w_K:=g_K-v_K$, which is equal to $u-p-v_K$ in $Q_{2R}$ by the choice of $K$. By \eqref{eq 1125.3} and \eqref{eq10.54}, $w_K$ satisfies
\begin{equation*}
\begin{cases}
 \partial_t w_K\le \cM^+w_K+h_K+\sup_{a\in \cA}\|f_a\|_{L_\infty}\quad &\text{in}\,\, Q_{2R},\\
 \partial_t w_K\ge \cM^-w_K+\hat{h}_K-\sup_{a\in \cA}\|f_a\|_{L_\infty}\quad &\text{in}\,\, Q_{2R},\\
 w_K=0\quad &\text{in}\,\, \bR^{d+1}_0\setminus Q_{2R},
\end{cases}
\end{equation*}
where
$$
h_K:= \cM^+(u-p-g_K),\quad \hat{h}_K:=\cM^-(u-p-g_K).
$$
By the dominated convergence theorem, it is not hard to see that
$$
\|h_K\|_{L_\infty(Q_{2R})},\,\,\|\hat{h}_K\|_{L_\infty(Q_{2R})}\to 0
\quad \text{as} \quad K\to \infty.
$$
We then fix $K$ large enough so that
$$
\|h_K\|_{L_\infty(Q_{2R})}+\|\hat{h}_K\|_{L_\infty(Q_{2R})}\le \sup_{a\in \cA}\|f_a\|_{L_\infty}.
$$
From Lemma \ref{lemma 12.271}, we have
\begin{equation}\label{eq 12.67}
\|w_K\|_{L_\infty(Q_{2R})}\le CR^\sigma\sup_{a\in \cA}\|f_a\|_{L_\infty},\quad
[w_K]_{\alpha/\sigma,\alpha;Q_{2R}}\le CR^{\sigma-\alpha}\sup_{a\in \cA}\|f_a\|_{L_\infty},
\end{equation}
where $C$ depends on $d,\sigma,\lambda$, and $\Lambda$.

Now let $q_K$ be the first-order Taylor expansion of $v_K$ at the origin.
Then by \eqref{eq 12.121}, for any $r\in (0,R/2)$,
 \begin{align} \nonumber
 &\|u-p-q_K\|_{L_\infty(Q_r)}
\le \|u-p-v_K\|_{L_\infty(Q_r)}+\|v_K-q_K\|_{L_\infty(Q_r)}\\ \nonumber
 &\le \|u-p-v_K\|_{L_\infty(Q_r)}
+Cr^{\sigma+\alpha} R^{-\sigma}[v_K]_{\alpha/\sigma,\alpha;(-R^\sigma,0)\times B_{2R}}\\
&\quad+Cr^{\sigma+\alpha}R^{-\alpha}\big([u]^t_{\Lambda^1}+[u]^\ast_{\sigma}+
 [Du]^t_{\frac{\sigma-1}{\sigma}}\big).
                                 \label{eq 12.66}
 \end{align}
Since $w_K=u-p-v_K$ in $Q_{2R}$,
we plug  \eqref{eq1.56} with $j=0$, 
 and \eqref{eq 12.67} to \eqref{eq 12.66} and obtain
\begin{align*}
\|u-p-q_K\|_{L_\infty(Q_r)}\le CR^\sigma\sup_{a\in \cA}\|f_a\|_{L_\infty}+Cr^{\sigma+\alpha}R^{-\alpha}\big([u]^t_{\Lambda^1}+[u]^\ast_{\sigma}
+[Du]^t_{\frac{\sigma-1}{\sigma}}\big).
  \end{align*}
Dividing both sides of the inequality above  by $r^{\sigma}$, we have
\begin{align*}
  &r^{-\sigma}\|u-p-q_K\|_{L_\infty(Q_r)}\\
  &\le C(R/r)^{\sigma}\sup_{a\in \cA}\|f_a\|_{L_\infty}+C(r/R)^{\alpha}
\big([u]^t_{\Lambda^1}
+[u]^\ast_{\sigma}+[Du]^t_{\frac{\sigma-1}{\sigma}}\big).
\end{align*}
Set $r=R/M$, where $M\ge 2$ is a constant to be determined. Note that the center of the cylinder can be replaced by any point $(t,x)$ in $\bR_0^{d+1}$, i.e.,
\begin{align}
                                    \label{eq7.28}
   &r^{-\sigma}\|u-p-q_K\|_{L_\infty(Q_r(t,x))}\nonumber\\
   &\le CM^{\sigma}\sup_{a\in \cA}\|f_a\|_{L_\infty}+CM^{-\alpha}
\big([u]^t_{\Lambda^1}+[u]^\ast_{\sigma}
+[Du]^t_{\frac{\sigma-1}{\sigma}}\big),
 \end{align}
which together with Lemma \ref{lem5.4} implies
 \begin{align}\nonumber
& [u]^t_{\Lambda^1}+[u]^\ast_{\sigma}
+[Du]^t_{\frac{\sigma-1}{\sigma}}\\
 & \le C  \sup_{r>0}\sup_{(t,x)\in \bR_0^{d+1}} r^{-\sigma}\inf_{p\in \mathcal{P}_1}\|u-p\|_{L_\infty(Q_r(t,x))}\nonumber\\
  &\le CM^{\sigma}\sup_{a\in \cA}\|f_a\|_{L_\infty}+CM^{-\alpha}
\big([u]^t_{\Lambda^1}+[u]^\ast_{\sigma}
+[Du]^t_{\frac{\sigma-1}{\sigma}}\big).\label{eq 2.43}
 \end{align}
 By taking $M$ sufficiently large in \eqref{eq 2.43} so that $CM^{-\alpha}<1/2$, we obtain
\begin{equation*}
[u]^t_{\Lambda^1}+[u]^\ast_{\sigma}+[Du]^t_{\frac{\sigma-1}{\sigma}}\le C\sup_{a\in \cA}\|f_a\|_{L_\infty}.
\end{equation*}

In the case when $\sigma=1$, by Lemma \ref{lem5.5b} and \eqref{eq 1126.1},
\begin{equation*}
[v_K]_{1+\alpha,1+\alpha;Q_{R/2}}\le C\big(
 R^{-1}[v_K]_{\alpha/\sigma,\alpha;(-R^\sigma,0)\times B_{2R}}+R^{-\alpha}[u]_{\Lambda^1}\big).
\end{equation*}
Then by the same proof, similar to \eqref{eq7.28}, we get
\begin{align*}
r^{-1}\|u-p-q_K\|_{L_\infty(Q_r(t,x))}\le CM\sup_{a\in \cA}\|f_a\|_{L_\infty}+CM^{-\alpha}[u]_{\Lambda^1},
\end{align*}
which together with Lemma \ref{lem5.4} gives
$$
[u]_{\Lambda^1}\le C\sup_{a\in \cA}\|f_a\|_{L_\infty}
$$
by taking $M$ sufficiently large.
The theorem is proved.
\end{proof}

Next, we provide a sketched proof of Corollary \ref{cor 1.3}.
\begin{proof}[Proof of Corollary \ref{cor 1.3}]
We only consider the case $\sigma>1$ and  divide the proof into three steps.

{\em Step 1.} For $k=1,2,\ldots$, denote $Q^k := Q_{1-2^{-k}}$.
Let $\eta_k\in C_0^\infty(Q^{k+1})$ be a sequence of nonnegative smooth cutoff functions satisfying
$\eta\equiv 1$ in $Q^{k}$, $|\eta|\le 1$ in $Q^{k+1}$,  $\|\partial_t^jD^i\eta_k\|_{L_\infty}\le C2^{k(i+j)}$ for each $i,j\ge 0$.  Set $v_k := u\eta_k$ to be a smooth function and notice that in $\bR^{d+1}_0$,
\begin{align*}
\partial_tv_k&=\eta_k \partial_tu+\partial_t\eta_k u
=\inf_{a\in \mathcal{A}}(\eta_k L_a u+\eta_k f_a+\partial_t\eta_k u)\\
&=\inf_{a\in \mathcal{A}}(L_a v_k+h_{ka}+\eta_k f_{a}+\partial_t\eta_k u),
\end{align*}
where
\begin{equation*}
h_{ka}=\eta_k L_a u-L_a v_k
=\int_{\bR^d}
\xi_k(t,x,y)K_a(y)\,dy,
\end{equation*}
and
\begin{equation*}
\xi_k(t,x,y) = u(t,x+y)(\eta_k(t,x+y)-\eta_k(t,x))-y^T D\eta_k(t,x)u(t,x).
\end{equation*}
Obviously, we have
\begin{align*}
\sup_{a\in \cA}\|\eta_k f_a\|_{L_\infty}\le \sup_{a\in \cA}\|f_a\|_{L_\infty(Q_1)}\quad\text{and}\quad
\|\partial_t\eta_k u\|_{L_\infty}\le C2^k\|u\|_{L_\infty(Q_1)}.
\end{align*}

{\em Step 2.} We  estimate the $L_\infty$ norm of $h_{ka}$.
By the fundamental theorem of calculus,
\begin{align*}
\xi_k(t,x,y) = y^T\int_{0}^1 u(t,x+y)D\eta_k(t,x+sy)-u(t,x)D\eta_k(t,x)\,ds.
\end{align*}
For $|y|\ge 2^{-k-3}$,
$$
|\xi_k(t,x,y)|\le C2^{k}|y|\|u\|_{L_\infty((-1,0)\times \bR^d)}.
$$
For $|y|<2^{-k-3}$, we can further write
\begin{align*}
\xi_k(t,x,y)& = y^T\int_{0}^1(u(t,x+y)-u(t,x))D\eta_k(t,x+sy)\\
&\quad +u(t,x)(D\eta_k(t,x+sy)-D\eta_k(t,x))\,ds,
\end{align*}
where the second term on the right-hand side is bounded by $C2^{2k}|y|^2|u(t,x)|$. To estimate the first term, we consider two cases: when $|x|\ge1-2^{-k-2}$, because $|y|<2^{-k-3}$, $\xi_k(t,x,y)\equiv 0$; when $|x|<1-2^{-k-2}$, we have
\begin{equation*}
\Big|y^T\int_0^1(u(t,x+y)-u(t,x))D\eta_k(t,x+sy)\,ds\Big|\le C2^{k}|y|^2\|Du\|_{L_\infty(Q^{k+3})}.
\end{equation*}
Hence for $|y|<2^{-k-3}$,
\begin{equation*}
|\xi_k(t,x,y)|\le C|y|^2\big(2^{2k}|u(t,x)|+2^{k}\|Du\|_{L_\infty(Q^{k+3})}\big).
\end{equation*}
Combining with the case when $|y|>2^{-k-3}$, we see that
\begin{equation*}
\|h_{ka}\|_{L_\infty}\le C2^{\sigma k}\|u\|_{L_\infty((-1,0)\times \bR^d)}+C2^{(\sigma-1) k}\|Du\|_{L_\infty(Q^{k+3})}.
\end{equation*}

{\em Step 3.} We apply Theorem \ref{thm 2} to $v_k$ and use the bounds in the previous steps to obtain
\begin{align*}
&[v_k]^t_{\Lambda^1}+[v_k]^\ast_\sigma+[v_k]^t_{\frac{\sigma-1}{\sigma}}
\le C\sup_{a\in \cA}\|f_a\|_{L_\infty(Q_1)}+C2^k\|u\|_{L_\infty(Q_1)}\\
& \quad
+C2^{\sigma k}\|u\|_{L_\infty((-1,0)\times \bR^d)}
+C2^{(\sigma-1) k}\|Du\|_{L_\infty(Q^{k+3})},
\end{align*}
where $C$ depends on $d$, $\lambda$, $\Lambda$, and $\sigma$, but independent of $k$. Since $\eta_k\equiv 1$ in $Q^k$,  it follows that
\begin{align}
		\nonumber
&[u]^t_{\Lambda^1(Q^k)}+[u]^\ast_{\sigma;Q^{k}}+[Du]^t_{\frac{\sigma-1}{\sigma};Q^k}\\
		\label{eq12.141}
&\le C2^{2 k}\|u\|_{L_\infty((-1,0)\times \bR^d)}+C2^{k}\|Du\|_{L_\infty(Q^{k+3})}+C\sup_{a\in \cA}\|f_a\|_{L_\infty(Q_1)}.
\end{align}
By the interpolation inequality, for any $\epsilon\in(0,1)$
\begin{align}
\label{eq12.142}
\|Du\|_{L_\infty(Q^{k+3})}
\le \epsilon[u]^\ast_{\sigma;Q^{k+3}}+C\epsilon^{-\frac {1} {\sigma-1}}\|u\|_{L_\infty(Q^{k+3})}.
\end{align}
Set $N = 1/(\sigma-1)$ and notice $N>1$.
Combining \eqref{eq12.141} and \eqref{eq12.142} with $\epsilon = C^{-1}2^{-k-10N}$, we obtain
\begin{align*}
&[u]^t_{\Lambda^1(Q^k)}+[u]^\ast_{\sigma;Q^k}
+[Du]^t_{\frac{\sigma-1}{\sigma};Q^k}\le C2^{2k+(k+10N)N}\|u\|_{L_\infty((-1,0)\times \bR^d)}\\
&\quad + 2^{-10N}[u]^\ast_{\sigma;Q^{k+3}}
+C\sup_{a\in \cA}\|f_a\|_{L_\infty(Q_1)}.
\end{align*}
Then we multiply $2^{-3kN}$ to both sides of the inequality above and get
\begin{align*}
	&2^{-3kN}([u]^t_{\Lambda^1(Q^k)}+[u]^\ast_{\sigma;Q^k}+[Du]^t_{\frac{\sigma-1}{\sigma};Q^k})\\
	&\le C2^{2(1-N)k}\|u\|_{L_\infty((-1,0)\times \bR^d)}+2^{-3N(k+3)-1}[u]^\ast_{\sigma;Q^{k+3}}
+C2^{-3kN}\sup_a\|f_a\|_{L_\infty(Q_1)}.
\end{align*}
We sum up the both sides of the inequality above and obtain
\begin{align*}
	&\sum_{k=1}^\infty2^{-3kN}([u]^t_{\Lambda^1(Q^k)}+[u]^x_{\sigma;Q^k}+[Du]^t_{\frac{\sigma-1}{\sigma};Q^k})\\
	&\le C\sum_{k=1}^{\infty}2^{2(1-N)k}\|u\|_{L_\infty((-1,0)\times \bR^d)}
+C\sum_{k=1}^\infty2^{-3kN}\sup_a\|f_a\|_{L_\infty(Q_1)}\\
	&\quad +\frac{1}{2}\sum_{k=4}^\infty2^{-3kN}([u]^t_{\Lambda^1(Q^k)}
+[u]^\ast_{\sigma;Q^{k}}+[Du]^t_{\frac{\sigma-1}{\sigma};Q^{k}}),
\end{align*}
which further implies that
\begin{align*}
&\sum_{k=1}^\infty2^{-3kN}([u]^t_{\Lambda^1(Q^k)}+[u]^\ast_{\sigma;Q^k}+[Du]^t_{\frac{\sigma-1}{\sigma};Q^k})\\
&\le C\|u\|_{L_\infty((-1,0)\times \bR^d)}+C\sup_a\|f_a\|_{L_\infty(Q_1)},
\end{align*}
where $C$ depends on $d$, $\lambda$, $\Lambda$, and $\sigma$. In particular,  by taking $k = 1$,  the corollary is proved.
\end{proof}

\section*{appendix}
In the appendix, we first provide a sketch of the proof of Corollary \ref{weak_harnack}.
\begin{proof}
By a scaling argument, we assume that $r=1$.  Let $k\ge 1$ be a constant to be determined later. Set $\hat{\delta}=\delta/k$. Let $(t_0,x_0)\in \overline{Q_{\delta/2}}$ be such that $u(t_0,x_0)=\inf_{Q_{\delta/2}}u$. Since $\sigma\in (1,2)$, we have $2^{-\sigma}\le 1-4^{-\sigma}$. By a scaling and translation of the coordinates, we apply Proposition \ref{thm 8.311} to $u$ in $Q_{\hat{\delta}}(t_0,x_0)$ and obtain
\begin{align*}
\hat\delta^{-(\sigma+d)/\varepsilon}\|u\|_{L_\epsilon(\hat{Q}_1)}
\le C_1\Big(\inf_{Q_{\hat{\delta}}(t_0,x_0)}u+C\hat{\delta}^\sigma \Big),
\end{align*}
where $\hat{Q}_1=Q_{\hat{\delta}}(t_1,x_0)$ and $t_1=t_0-(4^\sigma-1)\hat\delta^\sigma$. For any $x_1\in B_{\hat{\delta}/2}(x_0)$,
\begin{align*}
		\nonumber
&\|u\|_{L_\epsilon(\hat{Q}_1)}\ge \|u\|_{L_\epsilon((t_1-\delta^\sigma,t_1
)\times B_{\hat{\delta}/2}(x_1))}\\
&\ge C_2\hat{\delta}^{(\sigma+d)/\varepsilon}\inf_{(t_1-\delta^\sigma,t_1
)\times B_{\hat{\delta}/2}(x_1)}u\ge C_2\hat{\delta}^{(\sigma+d)/\varepsilon}\inf_{Q_{\hat\delta}(t_1,x_1)}u,
\end{align*}
where $C_2>0$ depending only on $d$. Therefore,
\begin{align*}
\inf_{Q_{\hat\delta}(t_1,x_1)}u\le C_1/C_2\Big(\inf_{Q_{\hat{\delta}}(t_0,x_0)}u+C\hat{\delta}^\sigma \Big)  .
\end{align*}
Applying Proposition \ref{thm 8.311} again, 
we have
\begin{align*}
\hat\delta^{-(\sigma+d)/\varepsilon}\|u\|_{L_\epsilon(\hat{Q}_2)}
\le C_1\Big(\inf_{Q_{\hat{\delta}}(t_1,x_1)}u+C\hat{\delta}^\sigma \Big),
\end{align*}
where $\hat{Q}_2=Q_{\hat{\delta}}(t_2,x_1)$ and $t_2=t_0-2(4^\sigma-1)\hat\delta^\sigma$,
and for any $x_2\in B_{\hat\delta/2}(x_1)$,
\begin{align*}
\inf_{Q_{\hat\delta}(t_2,x_2)}u\le C_1/C_2 \Big(\inf_{Q_{\hat{\delta}}(t_1,x_1)}u+C\hat{\delta}^\sigma \Big).
\end{align*}
By induction, or any $x_{n-1}\in B_{(n-1)\hat\delta/2}(x_0)\cap B_1$,
\begin{align*}
&\hat\delta^{-(\sigma+d)/\varepsilon}\|u\|_{L_\epsilon(\hat{Q}_n)}
\le C_3
\Big(\inf_{Q_{\hat{\delta}}(t_0,x_0)}u+C\hat{\delta}^\sigma \Big)\\
&\le C_3
\big(u(t_0,x_0)+C\hat{\delta}^\sigma \big)
=C_3
\Big(\inf_{Q_{\delta/2}}u+C\hat{\delta}^\sigma \Big),
\end{align*}
where $\hat{Q}_n=Q_{\hat{\delta}}(t_{n},x_{n-1})$, $t_{n}=t_0-n(4^\sigma-1)\hat{\delta}^\sigma$, and $C_3$ is a constant depending only on $\lambda$, $\Lambda$, $d$, and $n$. Notice that $|x_0|\le\delta/2$, $t_0\in[-(\delta/2)^\sigma,0]$, and $\sigma>1$. We can choose $k\ge 1$ in a suitable range depending only on $\sigma_2$ and $\delta$, and then  $n\le [2k/\delta]+1$, such that
$\hat Q_n$ runs through $(-\delta^\sigma,-\delta^\sigma+(4^\sigma-1)\hat{\delta}^\sigma)\times B_1$.
Finally, by applying Proposition \ref{thm 8.311} again and using a simple covering argument, we prove the corollary.
\end{proof}

Finally, we give the proofs of Lemmas \ref{lemma 12.31}, \ref{lem5.2}, and \ref{lem5.5}.

\begin{proof}[Proof of Lemma \ref{lemma 12.31}]
By mollification, it suffices to prove \eqref{eq5.01} assuming that $f\in C^\alpha((-1,0))$.
Let $x,y\in (-1,0)$, $y<x$, and $h:=x-y$. When $h\ge 1/3$,
$$
\frac{|f(x)-f(y)|}{h^\alpha}\le 2\cdot 3^\alpha\|f\|_{L_\infty((-1,0))}.
$$
When $h<1/3$, either $x<-1/3$ or $y>-2/3$. If $x<-1/3$, then $2x-y\in (x,0)$ and
\begin{align*}
\frac{|f(x)-f(y)|}{h^\alpha}&\le \frac{1}{2}\frac{|f(2x-y)+f(y)-2f(x)|}{h^\alpha}
+\frac{1}{2}\frac{|f(2x-y)-f(y)|}{h^\alpha}\\
&\le \frac{3^{\alpha-1}}{2}[f]_{\Lambda^1((-1,0))}
+\frac{1}{2^{1-\alpha}}[f]_{\alpha;(-1,0)}.
\end{align*}
The case when $y>-2/3$ is similar. Therefore,
\begin{equation*}
[f]_{\alpha;(-1,0)}\le 2\cdot 3^\alpha\|f\|_{L_\infty((-1,0))}+\frac{1}{2^{1-\alpha}}[f]_{\alpha;(-1,0)}
+\frac{3^{\alpha-1}}{2}
[f]_{\Lambda^1((-1,0))},
\end{equation*}
which yields \eqref{eq5.01}.
\end{proof}

\begin{proof}[Proof of Lemma \ref{lem5.2}]
First we consider the case when $\beta=1$. Integrating by part and  noting that $\eta''$ is an even function and $\int \eta''=0$, we obtain
\begin{align*}
&\big|\partial_t^2u^{(R)}\big|
=\Big|\int_{\bR}\big(2\partial_t^2 u(t-R^\sigma s,x)
-\partial_t^2 u(t-2R^\sigma s,x)\big)\eta(s-1)\,ds\Big|\\
&=\Big|\int_{\bR}R^{-2\sigma}
\Big(2u(t-R^\sigma s,x)-\frac 1 4 u(t-2R^\sigma s,x)\Big)
\eta''(s-1)\,ds\Big|\\
&=\Big|\int_{\bR}R^{-2\sigma}
\Big(u(t-R^\sigma s,x)+u(t-R^\sigma (2-s),x)-2u(t-R^\sigma ,x)\\
&\qquad -\frac 1 8
\big(u(t-2R^\sigma s,x)+u(t-2R^\sigma (2-s),x)-2u(t-2R^\sigma ,x)\big)\Big)
\eta''(s-1)\,ds\Big|\\
&\le CR^{-\sigma}[u]^t_{\Lambda^1}.
\end{align*}
For $\beta\in (0,1)$, when $r\ge R^\sigma$,
$$
r^{-1-\beta}|u^{(R)}(t,x)+u^{(R)}(t-2r,x)-2u^{(R)}(t-r,x)|
\le r^{-\beta}[u]^t_{\Lambda^1}\le R^{-\beta\sigma}[u]^t_{\Lambda^1}.
$$
When $r\in (0,R^\sigma)$, by \eqref{eq6.58} with $\beta=1$,
\begin{align*}
&r^{-1-\beta}|u^{(R)}(t,x)+u^{(R)}(t-2r,x)-2u^{(R)}(t-r,x)|
\le r^{1-\beta}[\partial_t u^{(R)}]^t_{1;\bR^{d+1}_0}\\
&\le Cr^{1-\beta}R^{-\sigma}[u]^t_{\Lambda^1}\le CR^{-\beta\sigma}[u]^t_{\Lambda^1}.
\end{align*}
From the above two inequalities, we immediately get \eqref{eq6.58}.
\end{proof}

\begin{proof}[Proof of Lemma \ref{lem5.5}]
We first estimate the H\"older semi-norm in $x$. By the interpolation inequality,
\begin{align}\nonumber
&[u-p]_{\alpha;(-R^\sigma,0)\times B_{2^jR}}^\ast\\
&\le (2^jR)^{-\alpha}\|u-p\|_{L_\infty((-R^\sigma,0)\times B_{2^jR})}+(2^jR)^{\sigma-\alpha}[u-p]^\ast_{\sigma;(-R^\sigma,0)\times B_{2^jR}}.\label{eq 12.43}
\end{align}
Because $p$ is linear,
\begin{equation}\label{eq 12.41}
[u-p]^\ast_{\sigma;(-R^\sigma,0)\times B_{2^jR}}=[u]^\ast_{\sigma;(-R^\sigma,0)\times B_{2^jR}}.
\end{equation}
Since $\eta$ has unit integral,  
we have
\begin{align}
&\big|u^{(R)}(t,x)-u(t,x)\big|\nonumber\\
&=\Big|\int_{\bR}\big(2u(t-R^\sigma s,x)-u(t-2R^\sigma s,x)
-u(t,x)\big)\eta(s-1)\,ds\Big|\le CR^\sigma[u]^t_{\Lambda^1}.
                    \label{eq10.31}
\end{align}
Furthermore,  for any $(t,x)\in (-R^\sigma,0)\times B_{2^jR}$,
\begin{align*}
&\big|u^{(R)}(t,x)-p(t,x)\big|
=\big|u^{(R)}(t,x)-u^{(R)}(0,0)-\partial_tu^{(R)}(0,0)t
-x^TDu^{(R)}(0,0)\big|\\
&\le \big|u^{(R)}(t,x)-u^{(R)}(t,0)-x^TDu^{(R)}(t,0)\big|\\
&\quad+\big|u^{(R)}(t,0)-u^{(R)}(0,0)-\partial_tu^{(R)}(0,0)t\big|
+\big|x^TDu^{(R)}(0,0)-x^TDu^{(R)}(t,0)\big|\\
&\le (2^jR)^\sigma[u]^\ast_{\sigma}+R^{2\sigma}
\big\|\partial^2_tu^{(R)}\big\|_{L_\infty(-R^\sigma,0)\times B_{2^jR}}+C2^jR^{\sigma}\big[Du^{(R)}\big]^t_{\frac{\sigma-1}{\sigma}}.
\end{align*}
Using Lemma \ref{lem5.2}, we have
\begin{align*}
\big\|u^{(R)}-p\big\|_{L_\infty((-R^\sigma,0)\times B_{2^jR})}\le C(2^jR)^\sigma[u]^\ast_{\sigma}+C2^jR^{\sigma}[Du]^t_{\frac{\sigma-1}{\sigma}}
+CR^\sigma[u]_{\Lambda^1}^t,
\end{align*}
which together with \eqref{eq10.31} implies that
\begin{equation}
\|u-p\|_{L_\infty((-R^\sigma,0)\times B_{2^jR})}\label{eq 12.42}
\le  C(2^jR)^\sigma[u]^\ast_{\sigma}+C2^jR^{\sigma}[Du]^t_{\frac{\sigma-1}{\sigma}}
+CR^\sigma[u]_{\Lambda^1}^t.
\end{equation}
We plug \eqref{eq 12.41} and \eqref{eq 12.42} in \eqref{eq 12.43} and get
\eqref{eq 12.62}.

Next we estimate the H\"older semi-norm in $t$. Obviously,
\begin{align*}
[u-p]^t_{\alpha/\sigma;(-R^\sigma,0)\times  B_{2^jR}}\le [u-p]^t_{\alpha/\sigma;(-2^{j\sigma/2}R^\sigma,0)\times B_{2^jR}}.
\end{align*}
From Lemma \ref{lemma 12.31} and scaling, we have
\begin{align*}
  &[u-p]^t_{\alpha/\sigma;(-2^{j\sigma/2}R^\sigma,0)\times B_{2^jR}}\\
  &\le C(2^{j/2}R)^{-\alpha}\|u-p\|_{L_\infty((-2^{j\sigma/2}R^\sigma,0)\times B_{2^jR})}+C(2^{j/2}R)^{\sigma-\alpha}[u-p]^t_{\Lambda^1}
  \\
  &\le C(2^{j/2}R)^{-\alpha}\|u-p\|_{L_\infty((-2^{j\sigma/2}R^\sigma,0)\times B_{2^jR})}+C(2^{j/2}R)^{\sigma-\alpha}[u]^t_{\Lambda^1}.
\end{align*}
We follow the proof of \eqref{eq 12.42} to estimate
\begin{align*}
  &\|u-p\|_{L_\infty((-2^{j\sigma/2}R^\sigma,0)\times B_{2^jR})}\\
  &\le\|u-u^{(R)}\|_{L_\infty((-2^{j\sigma/2}R^\sigma,0)\times B_{2^jR})}+\|u^{(R)}-p\|_{L_\infty((-2^{j\sigma/2}R^\sigma,0)\times B_{2^jR})}\\
  &\le CR^\sigma[u]^t_{\Lambda^1}+C(2^jR)^\sigma[u]^\ast_\sigma+C2^{j\sigma}R^\sigma
[u]^t_{\Lambda^1}+C2^{j(\sigma+1)/2}R^\sigma[D {u}]^t_{\frac{\sigma-1}{\sigma}}.
  \end{align*}
Therefore, we reach \eqref{eq 12.61}. The lemma is proved.
\end{proof}

\begin{proof}[Proof of Lemma \ref{lem5.5b}]
We modify the proof of Lemma \ref{lem5.5}. Similar to Lemma \ref{lemma 12.31}, we have
\begin{align}\nonumber
&[u-p]_{\alpha,\alpha;(-R,0)\times B_{2^jR}}
\le [u-p]_{\alpha,\alpha;Q_{2^jR}}\\
&\le (2^jR)^{-\alpha}\|u-p\|_{L_\infty(Q_{2^jR})}
+(2^jR)^{1-\alpha}[u-p]_{\Lambda_1;Q_{2^jR}}\nonumber\\
&\le (2^jR)^{-\alpha}\|u-p\|_{L_\infty(Q_{2^jR})}
+(2^jR)^{1-\alpha}[u]_{\Lambda_1}.\label{eq 12.43b}
\end{align}
Since $\eta$ and $\zeta$ have unit integral and $\zeta$ is radial,  
we have
\begin{align}
&\big|u^{(R)}(t,x)-u(t,x)\big|\nonumber\\
&=\Big|\int_{\bR^{d+1}}\big(2u(t-R s,x-Ry)-u(t-2R s,x-Ry)
-u(t,x)\big)\eta(s-1)\zeta(y)\,dy\,ds\Big|\nonumber\\
&=\Big|\int_{\bR^{d+1}}\big(2u(t-R s,x-Ry)-u(t-2R s,x-Ry)
-u(t,x-Ry)\big)\nonumber\\
&\qquad+\frac 1 2 \big(u(t,x+Ry)+u(t,x-Ry)
-2u(t,x)\big)\eta(s-1)\zeta(y)\,dy\,ds\Big|\nonumber\\
&\le CR[u]_{\Lambda^1}.
                    \label{eq10.31b}
\end{align}
For any $(t,x)\in Q_{2^jR}$, similar to Lemma \ref{lem5.2} with $\beta=\alpha/2$,
\begin{equation*}
\big|u^{(R)}(t,x)-p(t,x)\big|
\le (2^jR)^{1+\alpha/2}[u^{(R)}]_{1+\alpha/2,1+\alpha/2;Q_{2^jR}}
\le 2^{j(1+\alpha/2)}R[u]_{\Lambda^1},
\end{equation*}
which together with \eqref{eq10.31b} implies that
\begin{equation}
                                \label{eq 12.42b}
\|u-p\|_{L_\infty((-2^jR,0)\times B_{2^jR})}
\le  2^{j(1+\alpha/2)}R[u]_{\Lambda^1}.
\end{equation}
We plug \eqref{eq 12.42b} in \eqref{eq 12.43b} and get \eqref{eq 12.62b}.
The lemma is proved.
\end{proof}

\section*{Acknowledgments}
The authors would like to thank the referees for their careful review as well as many valuable comments and suggestions.


\bibliographystyle{plain}
\def\cprime{$'$}

\end{document}